\documentclass[english,colorinlistoftodos]{article}

\usepackage{amsmath,amsthm,amssymb,amsfonts, epsfig,mathtools,comment}
\usepackage{setspace, stmaryrd, verbatim, euro, enumerate, bbm}
\usepackage{authblk}

\usepackage{soul}

\usepackage{geometry}
\usepackage{xcolor}

\usepackage{todonotes}

\usepackage[colorlinks=true,linkcolor=blue,urlcolor=blue]{hyperref}

\newtheorem{theorem}{Theorem}[section]
\newtheorem{corollary}[theorem]{Corollary}

\newtheorem{lemma}[theorem]{Lemma}
\newtheorem{proposition}[theorem]{Proposition}

\theoremstyle{definition}

\newtheorem{definition}[theorem]{Definition}

\theoremstyle{remark}

\newtheorem{example}[theorem]{Example}

\newtheorem{remark}[theorem]{Remark}

\newcommand\cA{\mathcal{A}}
\newcommand\cB{\mathcal{B}}

\newcommand\cD{\mathcal{D}}

\newcommand\cF{\mathcal{F}}
\newcommand\cG{\mathcal{G}}
\newcommand\cH{\mathcal{H}}

\newcommand\cK{\mathcal{K}}

\newcommand\cM{\mathcal{M}}

\newcommand\cP{\mathcal{P}}

\newcommand\cS{\mathcal{S}}

\newcommand\bC{\mathbb{C}}

\newcommand\bG{\mathbb{G}}

\newcommand\bN{\mathbb{N}}

\newcommand\bP{\mathbb{P}}
\newcommand\bQ{\mathbb{Q}}
\newcommand\bR{\mathbb{R}}
\newcommand\bS{\mathbb{S}}

\newcommand{\ind}{\mathbbm{1}}

\newcommand{\ud}{\mathrm{d}}
\newcommand{\llb}{\llbracket}
\newcommand{\rrb}{\rrbracket}

\newcommand\fps{(\Omega, \cF, (\cF_t)_t,\bP)}

\makeatletter
\newcommand{\mylabel}[2]{#2\def\@currentlabel{#2}\label{#1}}
\makeatother

\title{An abstract decomposition of measures \\
	and its many applications}

\author[1]{Alessandro Milazzo \thanks{alessandro.milazzo@unito.it}}
\author[2]{Pietro Siorpaes \thanks{p.siorpaes@imperial.ac.uk}}
\affil[1]{School of Management and Economics, Dept.\ ESOMAS, University of Turin, Corso Unione Sovietica, 218 Bis, 10134, Turin, Italy.}
\affil[2]{ Dept.\ of Mathematics, Imperial College London, 16-18 Princess Gardens, SW71NE London, UK.}

\date{\today}

\numberwithin{equation}{section}

\begin{document}

	\maketitle
	
	\begin{abstract}
		We consider a little-known abstract decomposition result for positive  measures due to Dellacherie, and show that it yields many  decompositions of measures, several of which are new. We then extend Dellacherie's result to (controlled) vector measures, and apply it to obtain a decomposition of semimartingales due to Bichteler, on which we improve.
		Then, we investigate how the outputs of the decomposition  depend on its inputs, in particular characterising the two elements of the decomposition as projections in the sense of Riesz spaces and of metric spaces. Finally, we prove a decomposition theorem for  strictly positive operators on Riesz spaces which generalises Dellacherie's Theorem.
		
		\vspace{3pt}
		
		
	\end{abstract}

	\section{Introduction}
	In his book \cite[Th.\ T14, page 30]{Del72}, Dellacherie introduced the following theorem.
	\begin{theorem}\label{DellacherieThm}
		Let $\cF$ be a $\sigma$-algebra on $\Omega$, $\mu$ be a (positive and $\sigma$-additive)  $\sigma$-finite measure on $(\Omega,\cF)$, and $\cG\subseteq\cF$ be non-empty and closed under countable unions. Then there exist a set $\bar{G}\in\cG$ and two measures $\mu_\cG$ and $\mu_\cG^\perp$ such that
		\begin{equation}\label{DellacherieDecomp}
			\mu = \mu_\cG+\mu_\cG^\perp,
		\end{equation}
		where $\mu_\cG(\bar{G}^c)=0$, and $\mu_\cG^\perp(G)=0$ for every $G\in\cG$. Moreover, the decomposition \eqref{DellacherieDecomp} is unique.
	\end{theorem}

	We find this  decomposition result  to be very neat, and we agree with Dellacherie that\footnote{This being our translation of the original French text `qui m\'eriterait d'\^{e}tre classique'.} `it deserves to be classical'. 
	It is however a vast understatement to say that it has \emph{not} become classical. Indeed, we were not able to find any mention of it anywhere in the literature; even more curiously, Dellacherie himself, after introducing Theorem \ref{DellacherieThm},  never uses it once in his whole book, even if his \cite[Th.\ T41, page 58]{Del72} can easily be proved using Theorem \ref{DellacherieThm}, as we will show in our Corollary \ref{cor: stop time decomp}.
	
	The main contributions of this paper are as follows. We first show how Theorem \ref{DellacherieThm}, though easy to prove, can very profitably be applied to obtain many decompositions of measures, stopping times and processes, some well-known and some new; to this purpose, we also provide an extension of Theorem \ref{DellacherieThm} to vector measures and we show that it can be applied to the stochastic integral and to spectral measures. 
	
	We then  study of the dependence of $\bar{G}$ and $\mu_\cG$ on $\cG$ and $\mu$. In particular, we characterise the two elements of the decompositions, $\mu_\cG$ and $\mu_\cG^\perp$, as projections in the sense of Riesz spaces and of metric spaces. 
	Finally, we obtain a decomposition theorem for  strictly positive operators on Riesz spaces which generalises Theorem \ref{DellacherieThm}.
	
	We hope our effort may help pluck Dellacherie's nice theorem from obscurity. \\
	
	Here an outline of the paper.
	In Section \ref{se: proof of Dellacherie's theorem}, we feature Theorem \ref{thm: Dellacherie Thm Positive Meas} (a slight refinement of  Theorem \ref{DellacherieThm}). 
	In Section \ref{se: examples of decompositions},  we give many examples of decompositions of positive measures which can be obtained using Theorem \ref{thm: Dellacherie Thm Positive Meas}. 
	In Section \ref{se: decomposition vector measures},  we vastly extend the range of applicability of Theorem \ref{thm: Dellacherie Thm Positive Meas} by easily establishing a variant that holds for controlled vector measures, i.e., essentially all vector measures (including all Banach-valued and all $L^0$-valued ones), which we use to derive the Hahn-Jordan decomposition of a real measure. In Section \ref{se: stochastic integral}, we apply such extension to the stochastic integral with respect to a semimartingale, and obtain a decomposition of semimartingales due to Bichteler, but with a much more elementary proof and a sharper statement. 
	In Section \ref{se: the spectral measure}, we prove that any  spectral measure is controlled, so our decomposition theorem  can be applied to it. 
	In Section \ref{se: dependence of decomposition}, we investigate the monotonicity and order continuity of  $\bar{G}$ and $\mu_\cG$  as functions of $\cG$; and we show that the maps $\mu \mapsto \mu_\cG$, $\mu \mapsto \mu_\cG^\perp$, defined on the Banach lattice of real-measures,  are the projections in the sense of Riesz spaces (i.e., they are band projections) and of metric spaces (i.e., they are unique nearest point maps).
	Finally, in Section \ref{sec: dec in riesz spaces}, we state and prove a decomposition theorem for  strictly positive operators on Riesz spaces, which generalises Theorem \ref{DellacherieThm}.
	
	\section{Dellacherie's theorem for positive measures}
	\label{se: proof of Dellacherie's theorem}
	In this section we consider a theorem which slightly generalises Theorem \ref{DellacherieThm}.  This variant brings several advantages: (i) it provides a simple characterisation of $\bar{G}\in\cG$ which will prove useful in suggesting a proof and, more importantly, in enabling us to extend the decomposition theorem to vector measures in Section \ref{se: decomposition vector measures}; (ii) it allows us to consider finitely-additive (and not only countably-additive) measures; (iii) it outlines the case when there is a uniquely distinguished set $\bar{G}$ (the largest one) in its equivalence class\footnote{Notice that $\bar{G}$ in Theorem \ref{DellacherieThm} is always unique up to $\mu$-null sets, as it trivially follows from the uniqueness of  the decomposition \eqref{DellacherieDecomp}.}, which is important as this sometimes occurs in practice.
	Moreover, we point out that an alternative and essentially trivial proof of such decomposition result can be obtained by applying the concept of essential supremum. This has the considerable advantage that,  while Dellacherie's proof  relies on several non-trivial properties\footnote{These properties: (i) any family $M$ of measures bounded from above by $\mu$ admits a supremum  $\nu:=\sup M=\sup M_c$ for some countable  $\{\nu_k\}_{k\in \bN}=M_c \subseteq M$, where $\{\nu_k\}_{k\in \bN}$ can w.l.o.g.~be chosen to be increasing;  (ii) if $\{\nu_k\}_{k\in \bN}$ is  a sequence of measures then $(\sup_k \nu_k)(F)=(\sup_k \nu_k(F))$ holds \emph{if $\{\nu_k\}_{k\in \bN}$ is increasing}  (the first supremum being on the space of measures, the second on the set of real numbers).}
	of the (pointwise) order\footnote{I.e., $\nu \leq \mu$ if $\nu(F)\leq \mu(F)$ for all $F\in \cF$.} between measures (with no explicit mention or reference), the concept of $\mu$-essential supremum (i.e., of supremum on the space $L^0(\mu)$) is well known and its properties are easily proved.

	\medskip
	
	We now introduce some standard notations and definitions; in particular we recall the definition of $\mu$-essential supremum,  and an important property thereof.

	Given a  measurable space  $(\Omega,\cF)$,  $A^c$  denotes the complement of $A\subseteq \Omega$, i.e., $A^c:=\Omega\setminus A$. If $\lambda$ is a positive or a vector\footnote{I.e., a measure with values in a vector space, as defined in Section \ref{se: decomposition vector measures}.} measure (finitely or countably additive) on $(\Omega,\cF)$,  we will say that $F\in \cF$ is a \emph{$\lambda$-null set} (or a null set for $\lambda$) if the restriction $\lambda_{|F}:=\lambda(F \cap \cdot )$ of $\lambda$ to $F\in \cF$ is the zero measure, i.e., if $\lambda(G)=0$ for all $G \subseteq F, G\in \cF$; in this case we will say that $\lambda$ is \emph{concentrated} on $F^c$. Notice that if $\lambda$ is a  positive measure then $F\in \cF$ is $\lambda$-null if and only if $\lambda(F)=0$, so this  definition of null set extends the usual definition given for positive measures. 
	If a family of measures $(\lambda_i)_{i\in I}$ are concentrated on pairwise disjoint sets $(A_i)_{i\in I}$, we will say that they are (mutually) \emph{singular}. Given two positive measures $\mu$ and $\lambda$, we will say that $\mu$ is \textit{absolutely continuous} with respect to $\lambda$ if $\lambda(A)=0$ implies $\mu(A)=0$ for $A\in\cF$; in this case we write $\mu \ll \lambda$. If $\mu \ll \lambda$ and $\lambda \ll \mu$, we will say that $\mu$ and $\lambda$ are \textit{equivalent}, and write $\lambda\sim \mu$. From now on, when we say simply `measure', without further qualifications, we will mean a countably-additive positive measure, i.e., any countably-additive map $\mu:\cF\to[0,\infty]$ such that $\mu(\emptyset)=0$.

	\begin{remark}
		\label{rem: ae order} Let $\mu$ be a finite measure. The \emph{$\mu$-essential supremum} of a family $\cH$ of random variables is  the supremum (i.e., the smallest majorant) of $\cH$ seen as a subset of the Riesz space $L^0(\mu)$ of (equivalence classes of) random variables,  endowed with the order $\leq_\mu$ defined by: $X\leq_{\mu} Y$ if $X\leq Y $ $\mu$-a.e., i.e., if $\mu(\{X>Y\})=0$. In other words,  we say that a random variable $X$ is the $\mu$-essential supremum of a family $(X_i)_{i\in I}$ of random variables if 
		\begin{enumerate}
			\item  $X_i \leq X$ $\mu$-a.e.~for every $i\in I$,
			
			\item $X_i \leq Y$ $\mu$-a.e.~for every $i\in I$ implies $X \leq Y$ $\mu$-a.e..
		\end{enumerate} 
		It is easy to prove (see, e.g., \cite[Prop.\ 4.1.1]{EdSu92}) the existence of the $\mu$-essential supremum $X$ of any family $\cH \subseteq L^0(\mu)$ which is bounded from above\footnote{I.e., There exists $Y \in L^0(\mu)$ such that $H\leq Y$ $\mu$-a.e. for all $H\in \cH$.}, and the fact that $X$ is (the equivalence class of) $\sup_{n\in \bN} X_n$ for some sequence of random variables $X_n \in \cH$. It easily follows that  the $\mu$-essential supremum of a family of indicators (i.e., a $\{0,1\}$-valued random variable) is an indicator (see, e.g., \cite[Prop.\ 4.1.2]{EdSu92}), and so one can talk of $\mu$-essential supremum of a family of (measurable) sets. In other words, if we endow\footnote{To be precise $\leq_\mu$  is an order not really on $\cF$, but rather on the family $\cF^\mu$ of equivalence classes of sets in $\cF$ which differ by a $\mu$-null set.} $\cF$ with the order $\leq_\mu$ of inclusion $\mu$-a.e.~($A\leq_\mu B$ means $A \subseteq B $ $\mu$-a.e., i.e., $\mu(A \setminus B)=0$), then any family of sets $\cG \subseteq \cF$ admits a $\leq_\mu$-supremum $G$, and there exists a sequence $ G_n\in \cG, n\in \bN$ such that $\cup_n G_n =G$ $\mu$-a.e.. 
	\end{remark}

	\begin{remark}
		\label{}
		The (pointwise) supremum of an uncountable family of random variables is in general not measurable. So, in measure theory the concept of supremum of families of random variables has no importance; instead one always considers the $\mu$-essential supremum of families of \emph{equivalence classes} of random variables. The two concepts of course coincide when applied to countable families of random variables, since the countable union of sets of measure 0 has measure 0.
	\end{remark}

	\begin{theorem}
		\label{thm: Dellacherie Thm Positive Meas}
		Let $\mu$ be a finitely-additive positive measure on $(\Omega,\cF)$ and $\cG\subseteq\cF$
		. Then,
		\begin{enumerate}
			\item \label{Maximum set Dellacherie Decomp Equivalence} The following two statements are equivalent:
			
			\begin{enumerate}
				\item \label{Maximum set}  There exists a set $G_\mu\in\cG$ such that $G \setminus G_\mu$ is a $\mu$-null set for every $G\in \cG$.
				
				\item \label{DellacherieDecompVectMeas} There exist a set $\bar{G}\in\cG$ and two finitely-additive measures $\mu_\cG,\mu_\cG^\perp$ on $(\Omega,\cF)$ such that
				\begin{equation}
					\label{eq: decomp dell measures}
					\hspace{-1 cm} 	\mu = \mu_\cG+\mu_\cG^\perp, \text{ } \mu_\cG \text{ is concentrated on  } \bar{G}\in \cG,  \text{ every } G\in\cG \text{ is  $\mu_\cG^\perp $-null}.
				\end{equation}
			\end{enumerate} 
			Moreover, if the above statements hold then $G_\mu$ and $\bar{G}$ are unique up to a $\mu$-null set, $G_\mu=\bar{G}$ up to a $\mu$-null set, $\mu_\cG=\mu(G_\mu \cap \cdot)$ and $\mu_\cG^\perp=\mu(G_\mu^c \cap \cdot)$; in particular, such $\mu_\cG,\mu_\cG^\perp$ are unique.
			
			\item \label{exists max set} If $\mu$ is $\sigma$-finite and countably-additive, and $\cG$ is non-empty and closed under countable unions, then there exists a set $G_\mu\in\cG$ as in item \eqref{Maximum set}.
		\end{enumerate}
		
	\end{theorem}
	
	\begin{proof}
		We first prove item \ref{Maximum set Dellacherie Decomp Equivalence}. If \eqref{DellacherieDecompVectMeas} holds then $\bar{G}^c$ is $\mu_\cG$-null and $\bar{G}$ is $\mu^\perp_\cG$-null, and so $\mu_\cG=\mu(\bar{G} \cap \cdot)$ and $\mu_\cG^\perp=\mu(\bar{G}^c \cap \cdot)$, which also implies that $G \cap \bar{G}^c=G \setminus \bar{G}$ is $\mu$-null for every $G\in \cG$, i.e., \eqref{Maximum set} holds with $G_\mu:=\bar{G}$. Conversely, if $G_\mu$ is as in \eqref{Maximum set}, then \eqref{DellacherieDecompVectMeas} trivially holds with $\bar{G}:= G_\mu$, $\mu_\cG=\mu(\bar{G} \cap \cdot),\mu_\cG^\perp=\mu(\bar{G}^c \cap \cdot)$.
		Such $G_\mu$ is unique up to $\mu$-null sets, since if $\bar{G}_\mu\in\cG$ is also such that $G \setminus \bar{G}_\mu$ is $\mu$-null for every $G\in \cG$, then both $\bar{G}_\mu \setminus G_\mu$ and $ G_\mu \setminus \bar{G}_\mu $ are $\mu$-null.
		
		We now prove item \ref{exists max set}, using Remark \ref{rem: ae order}. If $\mu$ is a \emph{finite} measure then the $\mu$-essential supremum $G$ of $\cG$ exists; moreover, $G$  is the \emph{maximum} of $\cG$ under $\leq_\mu$ (indeed $G\in \cG$ because $\cG$ is closed under countable unions and there exists a sequence $ G_n\in \cG, n\in \bN$ such that $\cup_n G_n =G$ $\mu$-a.e.); in other words,  $G$ satisfies  item \eqref{Maximum set}.
		
		If $\mu$ is $\sigma$-finite but not finite then there exists $f>0$ such that $f\in L^1(\mu)$ (see, e.g., \cite[Lem.\ 6.9]{Ru06}). Therefore, we can apply the same argument as above to the finite positive measure $\nu=f \cdot \mu$, which is equivalent to $\mu$, and obtain the desired result.
	\end{proof}

	We want to point out that, though hardly apparent at first sight, behind the curtains our proof and Dellacherie's proof (\cite[Th.\ T14, page 30]{Del72}) are actually closely related, as follows. Consider w.l.o.g.~the case of finite $\mu$. While Dellacherie considers the order on the space $\cM$ of real measures, one can rely on the (easier to deal with) order on the space $L^1(\mu)$ of (equivalence classes of) $\mu$-integrable functions, because all measures considered in Dellacherie's proof are absolutely continuous with respect to $\mu$ (i.e., of the form\footnote{We use the notation $f\cdot\mu$ to denote the measure  $(f\cdot\mu)(F):=\int_F f d\mu$ for all $F \in \cF$.} $f \cdot \mu$), and so one can identify the measure $f \cdot \mu$ with the function $f$; moreover, the map $f \mapsto f \cdot \mu$, from $L^1(\mu)$ to $\{\nu \in \cM: \nu \ll \mu\}$ is an isomorphism of Banach lattices (see \cite[Th.\ 13.19]{AlBo99}).

	If $\mu, \cG,G_\mu$ are as in Theorem \ref{thm: Dellacherie Thm Positive Meas}, in analogy with the case discussed later in Corollary \ref{co: atomic decomp}, we will  use the following terminology:  we will  call $\mu_\cG$, $\mu_\cG^\perp$, and the equivalence class (up to equality $\mu$-a.e.) $\cG_\mu$ of $G_\mu$, respectively, the \emph{$\cG$-atomic part of $\mu$, the $\cG$-diffuse  part of $\mu$},  and the \emph{$\cG$-atomic support of $\mu$}. Notice that, as $G_\mu$ is unique up to $\mu$-null sets, $\cG_\mu$ is unique; as usual, we will often conveniently (though slightly inappropriately) disregard the distinction between $G_\mu$  and $\cG_\mu$, for example calling $G_\mu$ the $\cG$-atomic support of $\mu$, or writing $\cG_\mu$ to mean one of its representatives.

	\begin{remark}
		\label{rem: maximum set}
		As we will see, it sometimes happens that one is interested in a family of sets $\cG$ which  admits a maximum $G^*$ with respect to the \emph{pointwise} order $\leq $, i.e., the order of inclusion of sets (for which $A\leq B$ means $A \subseteq B$, i.e., the indicator functions satisfy $\ind_A \leq \ind_B$ at every point); equivalently, $G^*\in \cG$ satisfies $G \setminus G^*=\emptyset$ for all $G\in \cG$.
		In this case, though any set $\bar{G}$ which differs from $G^*$ by a $\mu$-null set will still satisfy \eqref{eq: decomp dell measures}, $G^*$ has the distinguished property of being the largest of the sets $\bar{G}$ such that \eqref{eq: decomp dell measures} holds (since the maximum is always unique). 
		
		Moreover, in this case we can apply item 1 of Theorem \ref{thm: Dellacherie Thm Positive Meas} to  conclude that $\mu_\cG,\mu_\cG^\perp$ as in \eqref{eq: decomp dell measures} exist even if $\mu$ is only finitely-additive. 
	\end{remark}

	\section{Examples of decompositions}
	\label{se: examples of decompositions}
	In this section we apply  Theorem \ref{thm: Dellacherie Thm Positive Meas} to obtain many decompositions of measures. We first highlight three well known decompositions, and a recently-discovered closely related one. Then we briefly list many other ones, several of which are new. The reader will notice how many choices of $\cG$ are of the form `all the subsets of $\Omega$ of size at most x', where the `size' could be meant in different ways: as measure, as dimension, as Baire category.
	
	\begin{corollary}[Lebesgue decomposition]\label{cor:LebDecomp}
		Let $\mu$ and $\nu$ be two $\sigma$-finite measures on $(\Omega,\cF)$. There exist two $\sigma$-finite measures $\mu^{ac}$, $\mu^s$ such that
		\begin{equation}\label{LebesgueDecomp}
			\mu=\mu^{ac}+\mu^s,
		\end{equation}
		where $\mu^{ac}$ is absolutely continuous with respect to $\nu$, and $\mu^s$ and $\nu$ are singular. Moreover, the decomposition \eqref{LebesgueDecomp} is unique. 
	\end{corollary}
	\begin{proof}
		Let $\cG:=\{A\in\cF: \nu(A)=0\}$ be the family of $\nu$-null sets, which contains $\emptyset$ and is closed under countable unions. Thus, Theorem \ref{thm: Dellacherie Thm Positive Meas} applied to $\mu$ and $\cG$ provides the unique decomposition $\mu=\mu_\cG+\mu_\cG^\perp$ where $\mu_\cG$ is concentrated on a $\nu$-null set (i.e., $\mu_\cG$ and $\nu$ are singular) and every $\nu$-null set is a $\mu_\cG^\perp$-null set (i.e., $\mu_\cG^\perp \ll \nu$). Hence, $\mu^{ac}:=\mu_\cG^\perp$ and $\mu^s:=\mu_\cG$ satisfy the required decomposition \eqref{LebesgueDecomp}. Finally, if \eqref{LebesgueDecomp} holds (i.e., $\mu=\mu^{ac}+\mu^s$ with $\mu^{ac}$ absolutely continuous with respect to $\nu$, and $\mu^s$ and $\nu$ singular), then $\mu^{ac}$ and $\mu^s$ satisfy, respectively, the properties that define $\mu_\cG^\perp$ and $\mu_\cG$, which are unique by Theorem \ref{thm: Dellacherie Thm Positive Meas}, showing the uniqueness of the decomposition \eqref{LebesgueDecomp}.
	\end{proof}

	Recall that a set $A\in\cF$ is an \emph{atom} of the measure $\mu$ if $\mu(A)>0$ and for every $B\in\cF$, $B\subseteq A$ either $\mu(B)=0$ or $\mu(B)=\mu(A)$; restated in the language of Remark \ref{rem: ae order}, the atoms of $\mu$ are the (equivalence classes of) minimal sets of $(\cF^\mu,\leq_\mu)$ which are not the minimum (i.e., which are not $\mu$-a.e.~$=\emptyset$).
	Note that if $A,B$ are  atoms of $\mu$, then either they coincide up to $\mu$-null sets, or they are disjoint up to $\mu$-null sets.
	We say that a measure $\mu$ is (\textit{purely}) \textit{atomic} if every measurable set of strictly positive measure contains an atom of $\mu$, and \emph{diffuse} (or \emph{non-atomic}) if it gives measure zero to every atom of $\mu$.
		 We refer to  \cite{jo70at} for further information on atomic and non-atomic measures.

\begin{corollary}[Decomposition into atomic and diffuse components]
		\label{co: atomic decomp}
		Let $\mu$ be a $\sigma$-finite measure on $(\Omega,\cF)$. 
	Then 
	\begin{enumerate}
		\item \label{it: atomic iff conc on countable}  $\mu$ is atomic if and only if it is concentrated on a countable union of atoms.
		\item  \label{it: atomic decomp}
		There exist two $\sigma$-finite measures $\mu^{a}$, $\mu^{d}$ such that
			\begin{equation}\label{AtomicDecomp}
			\mu=\mu^{a}+\mu^{d},
			\end{equation}
			where $\mu^{a}$ is atomic and $\mu^{d}$ is diffuse. Moreover, the decomposition \eqref{AtomicDecomp} is unique.
	\end{enumerate}	
\end{corollary}
\begin{proof}
	We now prove the `if' implication of item \ref{it: atomic iff conc on countable}. Note that if $(G_n)_{n\in \bN}\subseteq \cF$ are  atoms of $\mu$, disjoint up to $\mu$-null sets, then $ \mu(\cup_n (E\cap G_n))=\sum_n \mu(E\cap G_n)$ for every $E\in\cF$. Let $E\in \cF$, and assume that $\mu$ is concentrated on the countable union of atoms $\cup_n G_n$; then, 
	$$\mu(E)=\mu(E\cap (\cup_n G_n))=\mu(\cup_n (E\cap G_n)),$$ 
	so if $\mu(E)>0$ then $\mu(E\cap G_n)>0$ for some $n$, showing that $E$ contains the atom $E \cap G_n$, and thus $\mu$ is atomic.

		Let $\cG$ be the family of all the countable unions of atoms of $\mu$, which is trivially closed under countable unions. If  $\cG \neq \emptyset$, by applying Theorem \ref{thm: Dellacherie Thm Positive Meas}  to $\mu$ and $\cG$, we obtain the unique decomposition $\mu=\mu_\cG+\mu_\cG^\perp$ where $\mu_\cG$ is concentrated on a countable union of atoms  $\bar{G}\in\cG$; thus, $\mu_\cG$ is atomic (as proved above) and $\mu_\cG^\perp$ gives null measure to every atom (i.e., $\mu_\cG^\perp$ is diffuse), proving the existence 		 in item \ref{it: atomic decomp} with $\mu^{a}:=\mu_\cG$ and $\mu^{d}:=\mu_\cG^\perp$. If instead  $\cG = \emptyset$, taking $\mu^a:=0,\mu^{d}:=\mu$ proves their existence.

	Moreover, if   $\mu$ is  atomic then $\cG \neq \emptyset$ and $\mu_\cG^\perp=0$. Otherwise, by Theorem \ref{thm: Dellacherie Thm Positive Meas}, there exists $\bar G\in\cG$ such that $\mu(\bar{G}^c)\geq \mu_\cG^\perp(\bar{G}^c)=\mu_\cG^\perp(\Omega)>0$ and so $\bar{G}^c$ would contain a $\mu$-atom $G$; yet $\mu_\cG(G)=0$ (because $\mu_\cG$ is concentrated on  $\bar{G}$ and $\bar{G}\cap G=\emptyset$) and 
	$\mu_\cG^\perp(G)=0$ (because $\mu$ is diffuse), and so $\mu(G)=0$, contradicting the definition of atom. This shows that if  $\mu$ is  atomic then  $\mu=\mu_\cG$, which is thus concentrated on a countable union of atoms, proving the `only if' implication of item \ref{it: atomic iff conc on countable}.	
 
	Finally, assume that \eqref{AtomicDecomp} holds, i.e., $\mu=\mu^{a}+\mu^{d}$ with $\mu^a$ atomic and $\mu^d$ diffuse. By item \ref{it: atomic iff conc on countable}, $\mu^{a}$ is concentrated on a $\bar{G}\in \cG$. If  $\mu$ has no atoms (i.e., $\cG = \emptyset$) then $\mu^a$ is concentrated on $\emptyset$ and so $\mu^a=0$ and $\mu^d=\mu$; if instead $\cG \neq \emptyset$, then $\mu^{a}$ and $\mu^{d}$ satisfy, respectively, the properties that define $\mu_\cG$ and $\mu_\cG^\perp$, which are unique by Theorem \ref{thm: Dellacherie Thm Positive Meas}, showing the  uniqueness of the decomposition in \eqref{AtomicDecomp}.
\end{proof}
	
	\begin{remark}
		\label{re: smallest set of atoms}
		While Corollary \ref{co: atomic decomp} applies in any measurable space, a more precise statement can be given when each atom of $\mu$ is (the equivalence class of) a singleton or, more generally,  an atom of $\cF$. An atom of $\cF$ is defined as a set $A\in \cF$ such that $B\in\cF$, $B\subseteq A$ imply either $B=\emptyset$ or $B=A$; restated in the language of Remark \ref{rem: maximum set}, an atom of $\cF$ is a minimal set of $(\cF,\leq)$ which is not a minimum (i.e., which is non-empty).

		Indeed, in this case the set  of  atoms of $\cF$ with strictly positive measure $\mu$ is at most countable (because $\mu$ is $\sigma$-finite), and thus the unions of \emph{all} such sets is the pointwise maximum $G^*$ of the family $\tilde{\cG}$ of countable unions  of (some) such sets. This gives an example of a set $G^*$ which admits a maximum for the pointwise order, see Remark \ref{rem: maximum set}.
	\end{remark}

	We now show two applications of Theorem \ref{thm: Dellacherie Thm Positive Meas} to stopping and random times: the first one leads to the decomposition $\tau=\tau_A\wedge \tau_{A^c}$ of a stopping time $\tau$ into its accessible part $\tau_A$ and totally inaccessible part $\tau_{A^c}$; the second one leads to the decomposition $\tau=\tau_B\wedge\tau_{B^c}$ of a random time $\tau$ into its thin part $\tau_B$ and thick part $\tau_{B^c}$. For the former decomposition, we refer to \cite[Ch.~III, Sec.~2]{Pro04} where the definitions of accessible and totally inaccessible stopping times are given and the result is presented as Theorem 3. For the latter decomposition, we refer to \cite[Th.\ 5.5]{aksamit2021thin} and we will recall later  the non-standard definitions of thin and thick random times. Here, we fix an underlying filtered probability space $\fps$ and, for a set $A\in\cF$ and a stopping time $\tau$, we use the standard notations
	\begin{equation}
		\tau_A(\omega):=\begin{cases}
			\tau(\omega) & \text{if } \omega\in A\\
			\infty & \text{if } \omega\notin A.
		\end{cases}
	\end{equation}
	$$\llb \tau \rrb:=\{(\omega,t)\in\Omega\times[0,\infty): \tau(\omega)=t \}.$$ 
	and point out  that the identity
	\begin{align}
		\label{eq: graph tau_A}
		\llb \tau_A \rrb= 
		(A \times [0,\infty)) \cap \llb \tau \rrb 
	\end{align}
	trivially holds  for any stopping time $\tau$ and $A\in\cF$.
	\begin{corollary}
		\label{cor: stop time decomp}
		If $\tau$ is a stopping time, there exists  $A\in \cF$   such that $A\subseteq\{\tau<\infty \}$, $\tau_A$ is accessible, $\tau_{A^c}$ is totally inaccessible and $\tau=\tau_A\wedge\tau_{A^c}$ $\bP$-a.s.
	\end{corollary}
	
	\begin{proof}
		Since the constant $\infty$ is a predictable stopping time, the family
		$$\cG:=\{A\in\cF: A\subseteq\{\tau<\infty \} \text{ and } \tau_A \text{ is accessible} \}$$
		contains $\emptyset$, and it is not empty.
		Moreover, $\cG$ in closed under countable unions, as we now show. If $(A_n)_{n\in\bN}\subseteq\cG$ and $A=\cup_{n} A_n$, then \eqref{eq: graph tau_A} shows that  $\llb \tau_A \rrb = \cup_n \llb \tau_{A_n} \rrb$ and $A\subseteq\{\tau<\infty \}$. Since $\tau_{A_n}$ is accessible, there exists   predictable stopping times $(\tau_{n,m})_m$ such that
		$$\llb \tau_A \rrb = \cup_n \llb \tau_{A_n} \rrb\subseteq \cup_{n,m}\llb \tau_{n,m}\rrb \quad \bP\text{-a.s.},$$
		and so $\tau_A$ is accessible, thus $A\in\cG$. Hence, we can apply Theorem \ref{thm: Dellacherie Thm Positive Meas} to $\bP$ and $\cG$ and obtain $A:=G_\bP\in\cG$ such that $\bP(A^c\cap B)=0$ for every $B\in\cG$. Since  $A\in \cG$, $\tau_A$ is  accessible; let us prove that $\tau_{A^c}$ is totally inaccessible. Indeed, if $\sigma$ be a predictable stopping time and
		$$B:=A\cup\{\tau=\sigma<\infty \},$$
		then $\llb \tau_B \rrb \subseteq \llb \tau_A\rrb \cup \llb \sigma \rrb$ $\bP$-a.s., and so $B\in \cG$ and thus $\bP(A^c\cap B)=0$. Since
		$$B\cap A^c= \{\tau=\sigma<\infty \}\cap A^c= \{\tau_{A^c}=\sigma<\infty \},$$
		we get that
		$$0=\bP(A^c\cap B)=\bP(\{\tau_{A^c}=\sigma<\infty \}),$$
		and so $\tau_{A^c}$ is totally inaccessible. 
	\end{proof}

	\begin{remark}
		\label{rem: other proofs stop time decomp}
		As  suggested by the proofs of Corollary \ref{cor: stop time decomp}  given in \cite[Ch.\ 4, Th.\ 80,]{DeMeA} and \cite[Th.\ T41, page 58]{Del72}, there are
		other possible choices for the family $\cG$ to be chosen in the proof of Corollary \ref{cor: stop time decomp}. For example, if $\cS(\tau)$ denotes the set of increasing sequences $\left(\sigma_{n}\right)_n$ of stopping times bounded from above by $\tau$, and for  $\left(\sigma_{n}\right)_n\in \cS(\tau)$ we define
		$$ A\left[\left(\sigma_n\right)_n\right]:=\left\{\lim_{n} \sigma_{n}=\tau, \sigma_{n}<\tau \text { for all } n \right\} \cup\{\tau=0\}, $$ 
		then one can choose $\cG$ to be $\{A\left[\left(\sigma_n\right)_n\right]: \left(\sigma_{n}\right)_n\in \cS(\tau)\}$. Alternatively, one could take for $\cG$ the set of countable unions of sets of the form $\{\sigma=\tau<\infty\}$, where $\sigma$ ranges across all \emph{predictable} stopping times.
	\end{remark}

	To obtain an analogous  decomposition for random times, we recall the definitions of \emph{thin random time} and \emph{thick random  time}. A random time $\tau$ (i.e., a $[0,\infty]$-valued, $\cF$-measurable random variable) is thin if its graph is a thin set, i.e., if there exists a sequence of stopping times $(\tau_n)_{n=1}^\infty$ such that $\llb \tau \rrb\subseteq \cup_{n=1}^\infty \llb \tau_n\rrb$. A random time is thick if it avoids any stopping time $\sigma$, i.e., if $\bP(\tau=\sigma<\infty)=0$ for every stopping time $\sigma$.
	
	\begin{corollary}\label{cor:thin-thick}
		If $\tau$ is a random time, there exists $B\in\cF$ such that $B\subseteq\{\tau<\infty \}$, $\tau_B$ is thin, $\tau_{B^c}$ is thick and $\tau=\tau_B\wedge\tau_{B^c}$ $\bP$-a.s.
	\end{corollary} 
	
	\begin{proof} 
		Since the constant $\infty$ is a thin (and a thick) time, the family
		$$\cG:=\{B\in\cF: B\subseteq\{\tau<\infty \} \text{ and } \tau_B \text{ is thin} \}$$
		contains $\emptyset$, and so it is not empty. The rest of proof is  the same as that of Corollary \ref{cor: stop time decomp}, mutatis mutandis.
	\end{proof}
	
	\begin{remark}
		\label{}
		Of course we could, as it is done in the literature, prove that $A$  in Corollary \ref{cor: stop time decomp} is $\bP$-a.s. unique and belongs to $\cF_{\tau-}$, and that the decomposition of $\tau$ as $\alpha \wedge \beta$, where $\alpha$ is an accessible stopping time, $\beta$ is a totally inaccessible stopping time, and $\alpha \vee \beta=
		\infty$ $\bP$-a.s., is $\bP$-a.s.~unique; we do not do this, as we have nothing new to add in this regard. Analogously for Corollary \ref{cor:thin-thick}.
	\end{remark} 
	
	\begin{remark}
		Notice that, to obtain the decomposition of Corollary \ref{cor:thin-thick}, the authors in \cite[Sec.\ 5]{aksamit2021thin} need to prove some lemmas and theorems about dual optional projections, whereas we can more simply obtain it as a direct application of Theorem \ref{thm: Dellacherie Thm Positive Meas}.
	\end{remark}

	\medskip

	We will now list several more notable choices of $\cG$ in Theorem \ref{thm: Dellacherie Thm Positive Meas}, each one leading  to a  decomposition of a measure $\mu$ into its $\cG$-atomic and its $\cG$-diffuse part. In the following we assume, without repeating it, that $\cG$ is non-empty.

	\begin{enumerate}
		\item Let $\cF$ be the Borel $\sigma$-algebra on $\bR^n$ and $d_H$ be the Hausdorff dimension. Moreover, let $\cG:=\{G\in\cF: d_H(G)\in A\}$, where $A\subseteq [0,n]$ is a set closed under supremum. Then, the family $\cG\subseteq\cF$ is closed under countable unions due to the property of countable stability $d_H(\cup_k A_k)=\sup_k d_H(A_k)$ for all $A_k\in \cF$ of the Hausdorff dimension (see, e.g., \cite[shortly before Prop.\ 2.3]{Fal04}). In particular, taking $A=[0,n-1]$ we obtain that  the $\cG$-diffuse part of $\mu$ gives zero mass to every set of Hausdorff dimension at most $n-1$; such measures are considered in the theory of optimal transport, e.g.~\cite[Th.\ 2.12]{Vil03}, where such sets are called `small sets'.
		
		\item Let $\cF$ be the Borel $\sigma$-algebra on a locally compact Hausdorff space $\Omega$, then the family $\cG$ of Baire sets is closed under countable unions, and the $\cG$-atomic part of $\mu$ is the `Baire contraction' $\nu$ of  $\mu$ of (see \cite[Sec.\ 52, Th.\ H]{ha13}).
		\item Let $\cF$ be the Borel $\sigma$-algebra of a locally compact Hausdorff space $\Omega$, then the family $\cG$ of outer regular sets with respect to a measure $\mu$ is closed under countable unions (see \cite[Sec.\ 52, Th.\ C]{ha13}).
	\end{enumerate}

	Now we point out that there are two  operations which output suitable choices of $\cG$, i.e., a family of sets which is closed under countable unions. Then we list several choices of $\cG$ arising from the first, and then the second, such operation.
	
	\begin{enumerate}[(i)]
		\item Given a collection $\{\cG^i\}_{i\in I}$ of $\cG^i\subseteq \cF$ all closed under countable unions, then also $\cG:=\cap_i\cG^i$ is closed under countable unions. 
		\item Given a family $\cH\subseteq \cF$, then the family $\cG:=\cH^\sigma$ of all the countable unions of sets in $\cH$ is trivially closed under countable unions. 
	\end{enumerate}
	
	Here some choices of $\cG$ obtained via intersections.
	
	\begin{enumerate}
		\setcounter{enumi}{3}
		\item Let $\{\mu_i\}_{i\in I}$ be measures and $\cG^i$ be the family of $\mu_i$-null sets, then $\cG:=\cap_i\cG^i$ is the family of the polar sets of $\{\mu_i\}_{i \in I}$ (see, e.g., the definition in \cite[Sec.\ 1.1]{BoNu13}).
		\item Let $\nu$ be a measure on the Borel $\sigma$-algebra $\cF$ of a second-countable topological space $\Omega$, $\cG_1\subseteq\cF$ be the family of open sets and $\cG_2\subseteq\cF$ be the family of $\nu$-null sets. Then, by setting $\cG:=\cG_1\cap\cG_2$ we find that $\cG$ is closed under countable unions, and the second-countable assumption implies that $\cG$ admits a maximum $\bar{G}\in\cG$ for the pointwise order (see Remark \ref{rem: maximum set}); $\bar{G}^c$ is thus the topological support of $\nu$ (see, e.g., \cite[Sec.\ 12.3]{AlBo99}).
	\end{enumerate}
	\noindent Here some choices of $\cG$  obtained  taking countable unions. 
	\begin{enumerate}
		\setcounter{enumi}{5}
		\item Let $\cF$ be the Borel $\sigma$-algebra of a  topological space $\Omega$, and $\cH\subseteq\cF$ be the family of closed sets, then $\cG:=\cH^\sigma$ is the family of the so called $F_\sigma$ sets, i.e.,~countable unions of closed sets.
		\item Let $\cF$ be the Borel $\sigma$-algebra of a  topological space $\Omega$, and $\cH\subseteq\cF$ be the family of nowhere dense sets, then $\cG:=\cH^\sigma$ is the family of meagre sets (also called sets of first category; see, e.g.,~\cite[Sec.\ 3.11]{AlBo99}) .
		\item Let $\cF$ be the Borel $\sigma$-algebra on a metric space $\Omega$, $m\in\bN$ and $\cH\subseteq\cF$ be the family of $m$-rectifiable sets, then $\cG:=\cH^\sigma$ is the family of countably $m$-rectifiable sets (see, e.g., \cite[Sec.\ 3.2.14]{Fed14}). 
		\item 	Let $\cH$ be the collection of graphs of strictly decreasing Lipschitz functions, then the  family $\cG:=\cH^\sigma$ is closed under countable unions, so  the $\cG$-diffuse part of $\mu$ gives zero mass to the graphs of strictly decreasing Lipschitz functions. Such measures are considered in \cite[Th.\ 4.5]{KrYa19}. 
	\end{enumerate}

	\begin{remark}
		\label{re: uncountable unions}
		A natural generalisation of the construction $\cH^\sigma$ mentioned above would be, given any infinite cardinal number $o$, to consider the family $\cH^o$ of unions $\cup_{i\in I} H^i$ over a family of sets $H^i\in \cH, i\in I,$ where the index set $I$ has cardinality $|I|\leq o$. Indeed, $\cH^o$ is closed under countable unions (because any infinite set $I$ satisfies $|I \times I|=|I|$, see \cite[Th.\ 3.5]{Jec05}, and so $|I \times \bN | = |I|$), so in principle one could use it as a suitable choice for a family $\cG$. However, in general it will happen that $\cH^o$ contains non-measurable sets, since the $\sigma$-algebra $\cF$ is only closed under countable unions. For example, if  $\cH:=\{\{x\}: x\in \bR \}$ is the family of all singletons in $\bR$ then $\cH^{\aleph_1}$ is the family of \emph{all} subsets of $\bR$, which cannot be used to decompose a measure, whereas $\cH^{\aleph_0}=\cH^\sigma$ is the family of countable subsets of $\bR$, which leads to the decomposition into atomic and diffuse part of any measure defined on the Borel sets of $\bR$.
	\end{remark} 
	
	\medskip
	
	We have also come across the following instances where a family $\cG$ closed under countable unions, and which thus admits a maximum for $\leq_\mu$,  has been considered.
	
	\begin{enumerate}[(i)]
		\item In the proof of \cite[Lem.\ 2.3]{BrSc99}, the authors show that the family $\cB$ is closed under countable unions, and conclude that there is a set $\Omega_u$ of maximal measure in $\cB$,  which is unique up to $\bP$-null sets. Clearly $\Omega_u$ is the set $\bar{G}$ of Theorem \ref{thm: Dellacherie Thm Positive Meas} when $\mu=\bP$ and $\cG=\cB$.
		\item The proof of \cite[Prop.\ 2.4]{Kar12a} states that there exists a set $\Omega_w\in\cF$ satisfying the two dotted properties therein. 
		Clearly $\Omega_w$ is the set $\bar{G}$ of Theorem \ref{thm: Dellacherie Thm Positive Meas} when $\mu=\bP$ and 
		$$\cG:=\{G\in\cF: \bP(G\cap\{Y>0\})=0, \: \text{ for all }\: Y\in\cK\}:$$ 
		indeed, the first property immediately follows from the fact that $\bar{G}:=\Omega_w\in\cG$, and the second property follows from the fact that $\bP(A\cap\bar{G}^c)=0$ for all $A\in\cG$.
	\end{enumerate}
	
	\section{The decomposition theorem for vector measures}
	\label{se: decomposition vector measures}

	In this section we show how Theorem \ref{thm: Dellacherie Thm Positive Meas} can be extended to a very large class of vector measures. While, once identified the correct statement, the corresponding theorem is easy to prove, it is important because it will allow us to consider all of the most important and naturally occurring types of vector measures: the stochastic integral, the spectral measure, as well as any Banach-valued or $L^0(\bP)$-valued vector measure. 
	In turn, this will allow us to derive an interesting decomposition of semimartingales in Section \ref{se: stochastic integral}.
	We remark that while the notion of order between measures underlies Theorem \ref{thm: Dellacherie Thm Positive Meas}  in the case of real-measures, this is not so for vector measures, since (as far as we know) one cannot define (even for $\bR^2$-valued measures) an order on them such that two measures are singular with respect to such order if and only if they are concentrated on disjoint sets.
	
	We first introduce the notion of  vector measure on a measurable space $(\Omega,\cF)$.
	Let $V$ be a (Hausdorff)  Topological Vector Space (TVS); a $V$-valued  finitely-additive (vector) measure is a set function $\Theta:\cF\to V$ such that $\Theta(A\cup B)=\Theta(A)+\Theta(B)$ for any disjoint $A,B\in \cF$ (so in particular $\Theta(\emptyset)=0$).
	A $V$-valued  (vector) measure is a set function $\Theta:\cF\to V$  which is countably additive, i.e., such that the identity 
	\begin{equation}\label{VectMeasCountAdd}
		\Theta(\bigcup_{n=1}^\infty A_n)=\sum_{n=1}^\infty \Theta(A_n),
	\end{equation}
	holds for any sequence $A_n \in \cF, n\in\bN$ of pairwise disjoint sets, where $\sum_{n=1}^\infty \Theta(A_n)$ is the  limit of $\sum_{n=1}^k \Theta(A_n)$ as $k\to\infty$. 
	The simplest interesting case of a vector measure is obtained taking for $V$ the real line $\bR$, in which case we will speak of a \emph{real-measure}. Working component by component, one can use them to treat  $\bR^n$-valued, and complex-valued, measures.
	Of special interest are  Banach-valued measures, whose theory (both for finitely-additive and countably-additive measures) is presented  in \cite{DuSc57,DiestUhl:77}; aspects of the (much more complicated) case of measures with values in a general TVS are considered in  \cite{DrLa13}, whereas  \cite{Rol72} considers the important case of metric linear spaces (like $L^0(\bP)$).
	\begin{definition}
		\label{def: null set}
		We will say that a finite positive measure $\theta$ on $(\Omega,\cF)$ is a \emph{control} measure for a vector measure $\Theta$ on $(\Omega,\cF)$ if $\Theta \ll \theta$, i.e., if  every $\theta$-null set is a $\Theta$-null set. We will say that $\theta$ is equivalent to $\Theta$ if they have the same null sets.
	\end{definition} 
	We warn the reader that the definition of control measure can change slightly from one source to another. Though we will not need it, we should anyway mention the fact that every $\theta$-null set is a $\Theta$-null set is equivalent to asking that $\lim _{\theta(E) \rightarrow 0} \Theta(E)=0$ if $\Theta$ is Banach-valued (see \cite[Ch.\ 1, Sec.\ 2, Th.\ 1]{DiestUhl:77}); it is probably true in general, but we were not able to locate a corresponding reference.

	Given the definition of control measure, we can introduce a simple extension of Theorem \ref{thm: Dellacherie Thm Positive Meas} to vector measures. 
	\begin{theorem}\label{DellacherieThmVectMeas}
		Let $\cG \subseteq \cF$, and $V$ be any (Hausdorff) topological vector space. Then, 
		\begin{enumerate}
			
			\item \label{Maximum set Dellacherie Decomp Equivalence vector}
			Item \ref{Maximum set Dellacherie Decomp Equivalence} of Theorem \ref{thm: Dellacherie Thm Positive Meas} holds if the positive finitely-additive measures $\mu, \mu_\cG,\mu_\cG^\perp$  are replaced by finitely-additive $V$-valued (vector) measures  $\Theta, \Theta_\cG,\Theta_\cG^\perp$.
			
			\item If the  $V$-valued countably-additive measure  $\Theta$ admits a control measure $\theta$, and $\cG$ is non-empty and closed under countable unions, then there exists a set $G_\Theta\in\cG$  such that $G \setminus G_\Theta$ is a $\Theta$-null set for every $G\in \cG$ and,   up to $\Theta$-null sets, $G_\Theta$ is unique and equals $G_\theta$.
		\end{enumerate}

	\end{theorem}

	\begin{proof}   
		The proof of item \ref{Maximum set Dellacherie Decomp Equivalence vector} is absolutely identical to the case of positive finitely-additive measures. As for the other item, applying Theorem \ref{thm: Dellacherie Thm Positive Meas} to $\cG$ and the control measure $\theta$, we obtain a set $G_\theta\in\cG$ such that, for every $G\in \cG$,  $G_\theta^c \cap G=G \setminus G_\theta$ is a $\theta$-null set; as $\theta$ is a control measure for $\Theta$, by taking $G_\Theta:=G_\theta$, we obtain the existence of $G_\Theta$, which is trivially unique up to $\Theta$-null sets. 
	\end{proof}

	\begin{remark}\label{rmk:ContrMeas}
		Of course, for Theorem \ref{DellacherieThmVectMeas} to be of any use, we need to show that many vector measures admit a control measure; this is indeed the case, as we will now explain.
		The variation of a measure is an equivalent control measure for any real measure. More generally, 
		an equivalent control measure (with additional properties) exists for any Banach-valued measure: see e.g.~\cite[Ch.\ 1, Sec.\ 2, Cor.\ 6]{DiestUhl:77} or \cite[Ch.\ 4, Sec.\ 10, Lem.\ 5]{DuSc57}. 
		More generally still, \cite[Cor.\ 2]{musial1973absolute} shows that every measure with values in a metrizable locally-convex space admits an equivalent control measure.
		Moreover, even if the TVS $L^0(\bP)$  is neither locally convex nor locally bounded, it is true (though hard to prove) that any $L^0$-valued measure is convexly-bounded, and using this one can prove that any such $\Theta$ admits a control measure (see \cite[App.\ B, Th.\ B.2.2]{KwaWoy92}, or \cite[Th.\ B]{talagrand1981mesures}). This allows  to apply, in Section \ref{se: stochastic integral}, our Theorem \ref{DellacherieThmVectMeas}  to the special case  where the $L^0$-valued measure is the stochastic integral with respect to a  semimartingale (in this case  a control measure can actually be built somewhat explicitly, so its existence is easier to prove, as we  will explain in Section \ref{se: stochastic integral}). 
		Finally, we will prove in Section \ref{se: the spectral measure} that any spectral measure admits a control measure.
	\end{remark}
	
	\medskip

	To showcase the use of Theorem \ref{DellacherieThmVectMeas}, we will now use it to recover  Hahn-Jordan's decomposition of a real measure $\mu$; the proof is of course related to the classic proof found in \cite[Sec.\ 29]{ha13}, since essentially we break up their proof in parts, as follows. Our upcoming Lemma \ref{LemmaPositiveSets} states that, given any set $A$ of strictly positive measure, the family $A_{>}$ of its non-trivial subsets of $\mu$-positive measure is non-empty; since $A_{>}$ is closed under countable unions, Theorem \ref{DellacherieThmVectMeas} then shows that  $A_{>}$ has a maximum with respect to $\leq_\mu$ (see Remark \ref{rem: ae order}), from which it is then easy to prove Corollary \ref{co: Hanh-Jordan decomposition}. 
	
	Given a real measure $\mu$ on a measurable space $(\Omega,\cF)$, we  say that $A\in\cF$ is a \emph{positive set} for $\mu$ (or a $\mu$-\textit{positive} set) if, for every $B\in\cF$ such that $B\subseteq A$, we have $\mu(B)\geq 0$. Analogously, we say that $A\in\cF$ is a \emph{negative set} for $\mu$ if it is a positive set for $-\mu$.
	
	\begin{lemma}\label{LemmaPositiveSets}
		Let $\mu$ be a real measure on $(\Omega,\cF)$ and $A\in\cF$ such that $\mu(A)>0$. Then, there exists $\bar{A}\in\cF$ such that $\bar{A}\subseteq A$, $\bar{A}$ is $\mu$-positive, and is not $\mu$-null (i.e., $\mu(\bar{A})>0$).
	\end{lemma}
	
	\begin{proof}
		If $A_0:=A$ is positive then simply take $\bar{A}:=A_0$. Otherwise, we proceed by induction. If $A_{k-1},k\in \bN^*:=\bN \setminus \{0\}$ is not positive, let $n_k\in\bN^*$ be  the smallest integer for which there exists $B_k\in\cF$ such that $B_k\subseteq A_{k-1}$ and $\mu(B_k)<-1/n_k$. Choose one such $B_k$, and define $A_{k}:=A\setminus \cup_{i=1}^{k}B_i$; notice that $A_k \subseteq A_{k-1}$ implies $n_k\geq n_{k-1}$ for every $k\in \bN$. If there exists the smallest $K\in \bN$ such that $A_K$ is positive we can take $\bar{A}:=A_K$, and the thesis follows since the $\{B_k\}_k$ are disjoint and so
		\begin{align}
			\label{eq: mu A bigger}
			\textstyle 
			\mu(\bar{A})=\mu(A)-\sum_{k=1}^K \mu(B_k)>\mu(A)>0 .
		\end{align}
		Assume no $\{A_k\}_{k\in \bN}$ is positive; then the thesis follows taking 
		$\bar{A}:=A_{\infty}:=\cap_{k=0}^\infty A_k$, as we will now prove. Since $\{B_k\}_k$ are disjoint, the series $\sum_{k=1}^\infty \mu(B_k)$ converges, so  \eqref{eq: mu A bigger} holds with $K=\infty$, and $\lim_{k\to\infty}\mu(B_k)=0$ and thus  $\lim_{k\to\infty}1/n_k=0$. 
		To show that $\bar{A}$ is positive, assume by contradiction that there exists $N\subseteq\bar{A}$ such that $\mu(N)<0$. Since $\lim_{k\to\infty}1/n_k=0$,  there exists $k\in\bN$ such that $\mu(N)<-1/(n_k-1)$, and  since $N\subseteq\bar{A}\subseteq A_{k-1}$, the minimality property of $n_k$ is violated by $n_k-1$,  a contradiction.
	\end{proof}

	\begin{corollary}[Hahn-Jordan decomposition]
		\label{co: Hanh-Jordan decomposition}
		If $\mu$ is a real measure on $(\Omega,\cF)$,  there exists a $\mu$-positive set $\bar{G}\in \cF$ whose complement $\bar{G}^c$ is a $\mu$-negative set; such $\bar{G}$ is unique, up to $\mu$-null sets. In particular, setting 
		$$\mu^{+}:=\mu(\cdot \cap \bar{G}), \quad 
		\mu^{-}:=-\mu(\cdot \cap \bar{G}^c)$$
		one obtains 
		two finite positive mutually singular measures $\mu^+$ and $\mu^-$ such that $\mu=\mu^+-\mu^-$; such a decomposition of $\mu$ is unique.
	\end{corollary}
	\begin{proof} If $\mu$ or $-\mu$ is positive the theorem is trivially true, so let us assume otherwise.
		Let $\cG$ be the family of positive sets, which is closed under countable unions; by Lemma \ref{LemmaPositiveSets}, $\cG$ is non-empty. 
		Thus, Theorem \ref{DellacherieThmVectMeas} provides a positive set $\bar{G}\in\cG$ and  unique decomposition $\mu=\mu_\cG+\mu_\cG^\perp$ where $\mu^+:=\mu_\cG:=\mu(\cdot\cap\bar{G})$ and $\mu^-:=\mu_\cG^\perp:=\mu(\cdot \cap\bar{G}^c)$, such that  $\mu^-(G)=0$ for every $G\in\cG$. 
		Since  $\bar{G}$ is positive,  $\mu^+$ is a positive measure; we will now show that $\bar{G}^c$ is negative, concluding the proof of existence. 
		Assume by contradiction that there exists $A\in\cF$ such that $\mu^-(A)>0$ then, by Lemma \ref{LemmaPositiveSets}, there exists $\bar{A}\in\cF$ such that $\bar{A}\subseteq A$, $\bar{A}$ is positive (i.e., $\bar{A}\in\cG$) and $\mu^-(\bar{A})>0$, a contradiction.
		
		If $\bar{G}$ and $\tilde{G}$ are positive sets with negative complements, the sets 
		$\bar{G}\setminus \tilde{G}$ and  $\tilde{G}\setminus \bar{G}$ are  both positive and negative, and thus they are  $\mu$-null, proving the uniqueness of $\bar{G}$, and thus that of $\mu=\mu^+ - \mu^-$ (since a $\mu^-$-null  set on which $\mu^+$ is concentrated is $\mu$-positive and has $\mu$-positive complement). 
	\end{proof}
	
	\section{The stochastic integral and a decomposition of semimartingales}
	\label{se: stochastic integral}
	The most important example of a vector measure is represented by the stochastic integral with respect to a semimartingale, whose theory is developed, e.g., in the books \cite{Bi02}, \cite{Pro04}, \cite{DeMeB}. Such integral $I_X:H \mapsto (H \cdot X)_\infty$ is a $L^0(\bP)$-valued measure, since it satisfies the (stochastic) dominated convergence theorem; conversely, the celebrated theorem by Bichteler and Dellacherie (for an elementary proof see  \cite{BeSi14}) essentially states that if the integral with respect to a process $X$ is a $L^0$-valued measure,  then  $X$ is a semimartingale. Given a semimartingale $X$, the stochastic integral can also be considered as the map $J_X:H \mapsto (H \cdot X)$, which takes values in the space $\cS$ of semimartingales (so that $J_X(H)_\infty=I_X(H)$), and since $J_X$ also satisfies the dominated convergence theorem, it is also a ($\cS$-valued) vector measure.
	
	If $p\in [1,\infty)$, the previous statements admit the following much less well-known variant (see \cite[Th.\ 3.1]{yor79}): the stochastic integral $I_X$ with respect to $X$ is a $L^p$-valued measure if and only if  $X$ is a $\cS^p$-semimartingale.  In this case it is easy\footnote{By applying the Burkholder-David-Gundy inequality to the local martingale part $M$, this boils down to showing that $\ind_P \cdot [M]=0$ if and only if $\ind_P \cdot [M]^{\frac{1}{2}}=0$, where $P$ is a predictable set (see \cite[Th.\ 3.4]{yor79}). That this holds can be proved easily, since $t \mapsto t^{1/2}$ is strictly increasing on $[0,\infty)$; otherwise one can apply the more delicate \cite[Lem.\ 2.26]{LoObPrSi21}, since $t \mapsto t^{1/2}$ and $t \mapsto t^{2}$ are absolutely continuous on every $[a,b]\subseteq [0,\infty)$.} to see that
	$$
	\theta(\ud s\times \ud \omega):=(\ud [M]_s^{\frac 1 2}+\ud |A|_s)\ud\bP(\omega)
	$$
	is an equivalent control measure (for both $I_X$ and $J_X$, though now considered as having values in $L^p$ and $\cS^p$), where $X=M+A$ is the canonical decomposition of $X$. This allows to easily construct somewhat explicitly a control measure for an arbitrary  semimartingale $X$, as follows: by \cite[Ch.\ 4, Th.\ 34]{Pro04} there exists $\bQ\sim\bP$ such that $X$ is an $\cS^2(\bQ)$-semimartingale. Thus, $X$ can be canonically decomposed as $X=M+A$ under $\bQ$, and
	\begin{equation}\label{StochIntContrMeas}
		\theta(\ud s\times\ud \omega):=(\ud [M]_s^\frac{1}{2}+\ud |A|_s)\ud \bQ(\omega)
	\end{equation}
	is clearly an equivalent control measure for both $I_X$ and $J_X$.
	
	\medskip
	
	Among the possible applications of Theorem \ref{DellacherieThmVectMeas} to the stochastic integral, we show how to use it to obtain a decomposition that appears in \cite[Th.\ 4.4.5]{Bi02} in a more elementary way; we will also see how to obtain a slightly stronger result (see Remark \ref{rem: predictable support}). We need to introduce some definitions.  As usual, we always consider not sets, but rather their equivalence classes: any set $F \subseteq \Omega \times [0,\infty)$ is identified with the process $\ind_F$, and so sets $F,\tilde{F}$ are identified (and we write $F=\tilde{F}$) if $\ind_F=\ind_{\tilde{F}}$ a.s.~for all $t$ i.e., if the set $\{\ind_F \neq \ind_{\tilde{F}}\}$, which can more explicitly be written as 
	$$\hspace{-0.2 cm} 
	\{\omega: \ind_F(\omega,t) \neq \ind_{\tilde{F}}(\omega,t)\text{ for some } t\geq 0\}=\{\omega: (\omega,t)\in F \Delta \tilde{F}\text{ for some } t\geq 0\},$$
	is $\bP$-null; analogously $F\subseteq \tilde{F}$ actually means  that $\ind_F\leq \ind_{\tilde{F}}$ holds (i.e., it holds a.s.~for all $t$), and this we call the pointwise order for processes (and so for subsets of $\Omega \times [0,\infty)$). Given a semimartingale $X$, we can then consider the $X$-a.e. order given by $F\leq_X \tilde{F}$ if the semimartingale $ (\ind_{F \setminus \tilde{F}})\cdot X$ equals $0$, and identity sets $F,\tilde{F}$ which are equal $X$-a.e., i.e., such that $\{\ind_F \neq \ind_{\tilde{F}}\}$ is a $X$-null set.
	
	We will say that a set $F\subseteq \Omega \times [0,\infty)$  is \emph{thin} if $F=\cup_n\llb \tau_n \rrb$ for some stopping times $(\tau_n)_{n\in\bN}$; we call such $F$ a \emph{thin predictable}\footnote{This terminology is justified by the fact that a set $F$ is thin and predictable if and only if it is a  thin predictable set (we will not need this fact).} 
	set if $F=\cup_n\llb \tau_n \rrb$ for some predictable stopping times $(\tau_n)_{n\in\bN}$, and a
	\emph{thin totally-inaccessible} set if $F=\cup_n\llb \tau_n \rrb$ for some totally-inaccessible stopping times $(\tau_n)_{n\in\bN}$. We will use that the set of jumps $F:=\{\Delta X\neq 0 \}$ of a cadlag adapted process $X$ is  a thin set, and it is a  thin predictable set if $X$ is predictable (see, e.g., \cite[Prop.\ 1.32]{JaSh02} or \cite[Th.\ 3.1]{Sio13DM}).
	We will say that a cadlag process $X$ has \emph{predictable jump times} if $\{\Delta X\neq 0\}$ is thin predictable,  and that it has \emph{totally-inaccessible jump times} if $\{\Delta X\neq 0\}$ is thin totally-inaccessible (this is easily seen to be equivalent to $X$ being quasi-left-continuous).

	\begin{lemma}
		\label{le: single jump}
		If $Z$ is a semimartingale and $\tau$ a predictable stopping time then $\ind_{\llb\tau\rrb}\cdot Z=\Delta Z_{\tau}\ind_{[\tau, \infty)}$.
	\end{lemma}
	\begin{proof}
		Let $(\tau_{n})_{n\in \bN}$ be an increasing sequence of stopping times converging to $\tau$ and such that
		$\tau_{n}<\tau$ on $\{\tau>0\}$ for all $n$ (it exists: see \cite[Cor.\ 2.1]{Sio13DM}). Applying to $\tau_n$ and to $\tau$ the well known identity $\ind_{(0,\sigma]}\cdot Z=Z_{\sigma\wedge \cdot}$, valid for any stopping time $\sigma$, and subtracting the two we get $\ind_{(\tau_n,\tau]}\cdot Z=Z_{\tau\wedge \cdot}- Z_{\tau_{n}\wedge \cdot}$, and taking the limit as $n\to \infty$ by the stochastic dominated converge theorem we get the thesis.
	\end{proof} 
	
	\begin{corollary}\label{BichtDecompThm}
		Given a semimartingale $X$, there exists a thin predictable set $P$ such that 
		\begin{equation}\label{BichtelerDecomp}
			X=Y+Z,
		\end{equation}
		where $Y$ satisfies $Y=\ind_P \cdot Y$, and  $Z$ has totally-inaccessible jump times. Moreover, the decomposition \eqref{BichtelerDecomp} is unique.
	\end{corollary}
	\begin{proof}
		Let $\cG$ be the family of  thin predictable sets; this is obviously non-empty and closed under countable unions. Since the vector measure $J_X$ admits a control measure (recall \eqref{StochIntContrMeas}), we can apply Theorem \ref{DellacherieThmVectMeas} and obtain `the largest' thin predictable set $P:=G_{J_X}\in\cG$, such that $\ind_{P^c\cap G}\cdot X=0$ for every $G\in\cG$. Let $Y:=\ind_P\cdot X$, $Z:=\ind_{P^c}\cdot X$. 
		Since $\ind_P \ind_P=\ind_P$, and the stochastic integral satisfies $H\cdot (K \cdot X)=(HK) \cdot X$, we get
		$$Y=\ind_P \cdot X=\ind_P^2\cdot X=\ind_P\cdot (\ind_P\cdot X)=\ind_P\cdot Y;$$
		moreover, if $\tau$ is a \emph{predictable} stopping time then $\llb \tau\rrb\in\cG$ and so $\ind_{P^c\cap \llb\tau\rrb}\cdot X=0$. Thus,
		$$0=\ind_{P^c\cap \llb\tau\rrb}\cdot X=(\ind_{P^c}\ind_{\llb\tau\rrb})\cdot X=\ind_{\llb\tau\rrb}\cdot (\ind_{P^c} \cdot X)=\ind_{\llb\tau\rrb}\cdot Z .$$ 
		Lemma \ref{le: single jump} gives $\Delta Z_{\tau}\ind_{[\tau, \infty)}=0$, or equivalently $\Delta Z_{\tau}=0$, i.e.,  the set $J$ of  jumps of $Z$ satisfies $J \cap \llb\tau\rrb=\emptyset$. 
		Writing $J=\cup_n\llb \tau_n \rrb$ for some stopping times $(\tau_n)_{n\in\bN}$ we get $$\llb\tau_n\rrb \cap \llb\tau\rrb \subseteq J \cap \llb\tau\rrb=\emptyset \quad \text{ for all } n\in \bN,$$ 
		so every $\tau_n$ is totally-inaccessible, i.e., $J$ is thin totally-inaccessible.  The decomposition $X=Y+Z$ is unique, because given another one $X=\hat{Y}+\hat{Z}$ we get that the semimartingale $W:=Z-\hat{Z}=\hat{Y}-Y$ is quasi-left-continuous (and so $\ind_P \cdot W=0$), and satisfies $W=\ind_P \cdot W$, so $W=0$.
	\end{proof}
	
	\begin{remark}
		\label{rem: predictable support}
		The thin set $P\in\cP$ in Corollary \ref{BichtDecompThm} is  only unique up to $X$-null sets.
		However, if the semimartingale $X$ is predictable then $X=M+A$, with $M$ \textit{continuous} local martingale and $A$ predictable process of finite variation. So, the set of its jumps  is a  thin predictable set, which is obviously the smallest possible choice for the set $P$ (smallest with respect to the inclusion of sets; to be precise, we always consider equivalence classes of sets). 
		It turns out that even if $X$ is a local martingale, or more generally any semimartingale, there always exists the smallest choice $\bar{P}$ for a set $P$  as in  Corollary \ref{BichtDecompThm}, and this observation is (we believe) new. This choice involves the concept of predictable support, developed in \cite[Ch.\ 1, Sec.\ 2d]{JaSh02}, as we explain now. 
		
		Let $X$ be a general semimartingale, and consider its set of jumps $J:=\{\Delta X\neq 0 \}$. Write $J$ as  $J=\cup_n\llb \tau_n \rrb$ for some stopping times $(\tau_n)_{n\in\bN}$. By the well-known result of Corollary \ref{cor: stop time decomp}, each $\tau_n$ admits an a.s.~unique decomposition $\tau_n=\tau_n^a\wedge\tau_n^{ti}$ where $\tau_n^a$ is an accessible stopping time
		and $\tau_n^{ti}$ is a totally-inaccessible stopping time. 
		Therefore, $J=J^a\cup J^{ti}$, where $J^a:=\cup_n\llb\tau_n^a\rrb$ and $J^{ti}:=\cup_n\llb\tau_n^{ti}\rrb$ are disjoint (i.e., $\ind_{J^a \cap J^{ti}}=0$). The predictable support $J^p$ of the thin set $J$ is the smallest thin predictable set 
		that contains $J^a$, whose  existence is proved in \cite[Prop.\ 2.34]{JaSh02}. 
		Clearly $J^p$ is the $\leq$-smallest choice for the set $P$ in Corollary \ref{BichtDecompThm}, i.e., if $\bar{P}$ is another choice for $P$  then $J^p\subseteq \bar{P}$ a.s. for all $t$.
	\end{remark} 
	
	\begin{remark}
		\label{rem: decomp for marts and fin var}
		Applying Corollary \ref{BichtDecompThm} to a purely-discontinuous process of finite variation  $V=X$ we obtain two purely-discontinuous processes of finite variation $V^{jpr}:=Y,V^{jti}:=Z$, where $V^{jpr}$ has predictable jump times  and $V^{jti}$ has totally-inaccessible jump times, thus giving the decomposition \cite[Prop.\ 22.17]{Ka97}. Notice that, since a stopping time $\tau$ can be identified with the purely-discontinuous process of finite variation  $V=\ind_{(0,\tau]}$, which is predictable/accessible/totally inaccessible if and only if $\tau$ is such, this decomposition of $V$ subsumes the decomposition of $\tau$ that we obtained in Corollary \ref{cor: stop time decomp}. 
		
		Analogously, applying Corollary \ref{BichtDecompThm} to a purely-discontinuous local martingale $M=X$ gives two purely-discontinuous local martingales $M^{jpr}:=Y,M^{jti}:=Z$, where $M^{jpr}$ has predictable jump times and $M^{jti}$ has totally-inaccessible jump times, obtaining the decomposition in \cite[Prop.\ 23.16]{Ka97} (see also \cite[Th.\ 7.9]{Bi81}) as a corollary of Theorem \ref{BichtDecompThm}.
	\end{remark} 
	\section{The spectral measure}
	\label{se: the spectral measure}
	
	We now briefly introduce spectral measures; for more details we refer to \cite[Ch.\ 9, Sec.\ 1]{Con00}, or to \cite[Ch.\ 12, Sec.\ 17]{Ru91Fu} (which calls them \emph{resolution of the identity}). We then prove that they always admit an equivalent control measure, and thus Theorem \ref{DellacherieThmVectMeas} applies to them.
	
	Given a \emph{separable} Hilbert space $(\cH,(\cdot, \cdot)_\cH)$ on the field $\bC$ of complex numbers, let $\cB(\cH)$ be the set of bounded linear transformations from $\cH$ into $\cH$. Given a measurable space $(\Omega,\cF)$, a \emph{spectral measure} for $(\Omega,\cF,\cH)$  is a function $E: \cF \rightarrow \cB(\cH)$
	such that:
	\begin{enumerate}
		\item   for each $\Delta\in \cF, E(\Delta)$ is an orthogonal projection on a closed subspace of $\cH$;
		
		\item $E(\emptyset)$ equals the projection $0$ on the origin $\{0_\cH\}$, and $E(\Omega)$ equals the identity $1=Id_\cH$ on $\cH$;
		
		\item \label{commute} $E\left(\Delta_{1} \cap \Delta_{2}\right)=E\left(\Delta_{1}\right) E\left(\Delta_{2}\right)$ for each $\Delta_{1}$ and $\Delta_{2}$ in $\cF$;
		\item \label{spectral measure additive}
		For every $x \in H$ and $y \in H$, the set function $E_{x, y}:\cF \to \bC$ defined by
		$$
		E_{x, y}(\cdot):=(E(\cdot) x, y)_\cH
		$$
		is a complex measure on $\cF$.
	\end{enumerate} 
	Let us now point out the relationships between spectral measures and vector measures. 
	Item \ref{spectral measure additive} states that $E_{x,y}$ is a $\bC$-valued measure, i.e., if $\left\{\Delta_{n}\right\}_{n\in \bN}$ are pairwise disjoint sets of $\cF$ then
	\begin{align}
		\label{eq: series Exy}
		E_{x,y}\left(\bigcup_{n\in \bN} \Delta_{n}\right)=\sum_{n\in \bN} E_{x,y}\left(\Delta_{n}\right).
	\end{align}
	It follows\footnote{Because  $E_{x,y}$ is a complex measure, and item \ref{commute} implies that the projections $E\left(\Delta_{1}\right)$ and $E\left(\Delta_{2}\right)$ commute (which is equivalent to saying that they are orthogonal, i.e., their ranges are orthogonal subspaces).} 
	that, for each $x\in \cH$, $\Delta \mapsto E(\Delta) x$ is a (countably additive) $\cH$-valued (vector) measure on $(\Omega,\cF)$.
	Notice that $E$ is \emph{not} a vector measure with values in the Banach space $\cB(\cH)$, since given pairwise disjoint sets $\left\{\Delta_{n}\right\}_{n\in \bN}$ of $\cF$, the series
	\begin{align}
		\label{eq: series E}
		\textstyle
		\sum_{n\in \bN} E\left(\Delta_{n}\right)
	\end{align}
	does \emph{not}\footnote{It only converges if all but finitely many of the $E\left(\Delta_{n}\right)$ are 0 (since the norm of any projection is either 0 or 1, the partial sums of the series  in \eqref{eq: series E} cannot form a Cauchy sequence).} generally converge \emph{in the norm topology} of $\cH$. On the other hand,  \eqref{eq: series Exy} states that the series in \eqref{eq: series E} converges in the  Strong Operator Topology\footnote{The SOT is the topology on $\cB(\cH)$ generated by the family of seminorms $\{p_x\}_{x\in \cH}$, where $p_x(T):=||Tx||_{\cH}$ for $T\in \cB(\cH)$.} (SOT); thus $E$ is a (special type of) $\cB(\cH)$-valued measure, if $\cB(\cH)$ is endowed with the SOT (as such, it is a locally-convex TVS).
	
	The definitions of null set, concentration, mutual singularity, and control measure, for a spectral measure $E$ on on $(\Omega,\cF,\cH)$  are the same as for a vector measure; in particular, $F\in \cF$ is a \emph{$E$-null set} if $E(G)=0$ for all $G \subseteq F, G\in \cF$, and a finite positive measure $\mu$ is a \emph{control} for $E$ if every $\mu$ -null set is a $E$-null set.

	\begin{lemma}
		\label{le: spectral control}
		Any spectral measure $E$ admits an equivalent control measure, and $\Delta\in \cF$ is a $E$-null set if and only if $E(\Delta)=0$.
	\end{lemma}
	\begin{proof}
		If $E$ is a spectral measure for $(\Omega,\cF,\cH)$, and $\Delta \in \cF$,  each $E(\Delta)$ is an orthogonal projection, so $x-E(\Delta)x$ is orthogonal to $E(\Delta)x$ for each $x\in \cH$, leading to the identity 
		\begin{align}
			\label{eq: spectral positive}
			||E(\Delta)x||^2_\cH=E_{x,x}(\Delta) , \quad x\in \cH .
		\end{align}
		Equation  \eqref{eq: spectral positive} shows that the complex measure $E_{x,y}$ is  a finite positive measure when $y=x$, and that $\Delta$ is a $E$-null set if and only if $E(\Delta)=0$, i.e., if and only if $E_{x,x}(\Delta)=0$ for all $x\in \cH$. 
		Let $D:=\{x_n\}_{n\in \bN}$ be a dense subset of $\cH$. Then, 
		since $(x,y)\mapsto E_{x,y}(\Delta)$ is continuous, $E_{x,x}(\Delta)$ equals 0 for all $x\in \cH$ if and only if $ E_{x_n,x_n}(\Delta)=0$ for all $n\in \bN$.
		Since $E_{x,y}$ has total variation at most $\|x\|_{\cH}\|y\|_{\cH}$ (see \cite[Ch.\ 9, Sec.\ 1, Lem.\ 1.9]{Con00}), the formula
		$$  \mu:=  \sum_{n\in \bN} t_n E_{x_n,x_n} , 
		\quad \text{ where } t_n:= \frac{2^{-(n+1)}}{1+\|x_n\|_{\cH}^2}>0 ,$$
		defines a positive \emph{finite} measure (of mass at most 1). Since $\mu(\Delta)=0$ if and only if $ E_{x_n,x_n}(\Delta)=0$ for all $n\in \bN$, $\mu$ is an equivalent control measure for $E$.
	\end{proof} 
	
	\begin{remark}
		\label{rem: DellacherieThmSpectralMeas}
		It follows from Lemma \ref{le: spectral control}  that, if $E$ is a spectral measure for $(\Omega,\cF,\cH)$, then Theorem \ref{DellacherieThmVectMeas} applies to $\Theta:=E$; of course the resulting $\Theta(G_\Theta\cap \cdot )$ and $\Theta(G_\Theta^c\cap \cdot )$ are \emph{not} spectral measures, since they do not satisfy the property that, once applied to  $\Omega$, they equal the identity (i.e., the projection on the whole of $\cH$). They do however satisfy all the other properties which define a spectral measure.
	\end{remark}
	
	\begin{remark}
		\label{rem: spectral theorem}
		The crucial result about spectral measures, which clarifies their importance, is the spectral theorem, which `can be used to answer essentially every question about normal operators' \cite[Page 255]{Con00}. 
		
		If $N$ is a normal operator on a separable Hilbert space $\cH$, and $E^N$ is the  associated spectral measure given by the spectral theorem, then the fact that $E^N$ admits a control measure $\mu$ is stated and proved in \cite[Ch.\ 9, Prop.\ 8.3]{Con00}. We chose to present Lemma \ref{le: spectral control} anyway, since  it is much more elementary.
	\end{remark}

	\section{Dependence of    $\mu_\cG$ and $\cG_\mu$  on  $\mu$ and $\cG$}
	\label{se: dependence of decomposition}
	In this section we study  the dependence on   $\cG\subseteq \cF$ (a non-empty family closed under countable unions) and $\mu$ (a real measure) of $\cG_\mu$ (the $\cG$-atomic support of $\mu$),    $\mu_\cG$ and $\mu_\cG^\perp$ (the $\cG$-atomic  and $\cG$-diffuse parts of $\mu$).

	Let $\cM$ be the space of signed measures (i.e., real-valued measures) on $(\Omega,\cF)$ and $\cM_+\subseteq\cM$ be the convex cone of positive finite measures. Recall that the existence of $\cG_\mu$, $\mu_\cG$ and $\mu_\cG^\perp$ for any $\mu\in\cM$ is guaranteed by Theorem \ref{DellacherieThmVectMeas} and the fact that the variation of $\mu$ is a control measure for $\mu$ (see Remark \ref{rmk:ContrMeas}).
	
	We endow $\cM$ with the set-wise (partial) order $\geq$  and with the total variation norm $\|\cdot\|_\cM$ defined, for $\mu,\nu\in\cM$, as $\mu\geq\nu$ if $\mu(A)\geq \nu(A)$ for every $A\in\cF$ (i.e., we take $\cM_+$ as the positive cone) and $\|\mu\|_\cM:=|\mu|(\Omega)$, where $|\mu|\in \cM_+$ is the variation of $\mu\in \cM$. We recall  that $(\cM,\geq, \| \cdot \|_\cM)$ is an AL-space (see \cite[Th.\ 9.51 and 10.56]{AlBo99}), and it is order complete (\cite[Th.\ 10.53, 10.56, Lem.\ 8.14]{AlBo99}); in particular  $|\mu|=\mu \vee (-\mu)$, and there exists the supremum $\vee_{i\in I} \mu_i$ of any net $\{\mu_i \}_{i\in I}\subseteq\cM$ bounded from above; moreover, \cite[Th.\ 10.53, item (2)]{AlBo99} shows that \emph{if   $\{\mu_i \}_{i\in I}$ is increasing} then 
	\begin{equation}\label{IncreasingNet}
		(\vee_{i\in I}\mu_i)(A)=\sup_{i\in I}(\mu_i(A)), \quad A\in\cF ;
	\end{equation}
	obviously \eqref{IncreasingNet} fails without the monotonicity assumption, since for general measures $\nu_1,\nu_2$ the supremum $\nu_1 \vee \nu_2:=\sup\{\nu_1,\nu_2\}$ satisfies
	$$(\nu_1 \vee \nu_2)(F)=
	\sup\{\nu_1(B)+\nu_2(F\setminus B): \: B\in\cF, \: B\subseteq F \},$$
	and so does not normally equal $\sup\{\nu_1(F),\nu_2(F)\}$.
	Define the set of $\cG$-atomic measures to be 
	\begin{align}
		\label{eq: atomic measures}
		\cA(\cG):=\{\mu \in \cM:  \text{there exists } G\in \cG \text{ such that $G^c$ is $\mu$-null}\},
	\end{align}
	so that $\mu \in \cM$ belongs to  $\cA(\cG)$ if and only if $\mu$ is concentrated on some $G\in \cG$. 
	Then, define  the set of $\cG$-diffuse measures to be
	\begin{align}
		\label{eq: diffuse measures}
		\cD(\cG):=\{\mu \in \cM: \text{ every } G\in \cG \text{ is   $\mu$-null}\},
	\end{align}
	and recall that  $F\in \cF$ is  $\mu$-null iff $|\mu|(F)=0$. 
	
	We now recall some facts about projections, in three different contexts. A linear operator $P$ on a vector space $B$ is called a projection if $P^2=P$, i.e., if $Pb=b$ for every $b$ in the range of $P$. Two subspaces $V,W$ of  $B$ are said to be complemented if $V+W=B$ and $V \cap W=\{0\}$, in which case we write $B=V \oplus W$, and each $b\in B$ decomposes as $v+w$ for a unique $(v,w)\in V\times W$, which defines the projections $b \mapsto v$ and $b \mapsto w$ of $B$ on $V$ and $W$.
	We say that $x$ in a Banach space $(B,\| \cdot \|_B)$    is  orthogonal to $y\in B$ (in the sense of James, see \cite{jam45,jam47}), and write $x \perp y$, if $\|x\|_B=\inf \{\|x+t y\|_B: t \in \bR\} ;$ if $C,D \subseteq B$, then we say that they are orthogonal, and write $C \perp D$, if $c\perp d$ for each $c\in C, d \in D .$
	Given a subset $C$ of a metric space $(B,d_B)$, if for every $b \in B$ there exists a unique  minimizer $P_C(b)$ of the distance $C \ni c \mapsto d_B(b,c)$ from $C$, then $P_C: B \to C$ is called the \emph{unique nearest-point map} of $C$ (often called the \emph{metric projection} on $C$), and $C$ is said to be a Chebyshev set.
	
	We also introduce some concepts for Riesz spaces, for which we refer to \cite{AlBo99, aliprantis2006positive} for further details.
	Given a partially ordered set $(E,\geqslant)$ and $A\subseteq E$, we say that $s\in E$ is the \emph{supremum} of $A$ (written $s=\sup A$) if $s\geq a$ for all $a\in A$, and  $m\geq a$ for all $a\in A$ imply $m\geq s$, and analogously define the \textit{infimum}; we say that $E$ is a \textit{lattice} if each pair of elements $x,y\in E$ has a supremum $x\vee y:=\sup\{x,y\}$ and an infimum $x\wedge y:=\inf\{x,y\}$. A \textit{Riesz space} is an ordered vector space that is also a lattice, and its positive cone is $E_+:=\{x\in E: x\geq 0\}$. In what follows (before Proposition \ref{th: dependence atomic part}), $(E,\geqslant)$  will always denote a Riesz space and $A\subseteq E$ a subset of $E$.
	For a vector $x\in E$, its \textit{absolute value} $|x|$, \emph{positive part} $x^+$ and \emph{negative part} $x^-$ are defined by $$|x|:=x\vee(-x), \quad x^+:=x\vee 0, \quad x^-:=(-x)\vee 0,$$
	and satisfy $x=x^+ - x^-, |x|=x^+ + x^-$.  A set $A\subseteq E$ is called \textit{solid} if $|x|\geqslant|y|$ and $x\in A, y\in E$ imply $y\in A$. A solid vector subspace of $E$ is called an \textit{ideal} of $E$.
	
	We say that $A$ is \textit{order bounded from above} (respectively, \textit{from below}) if there exists $x\in E$ such that $x\geqslant a$ (respectively, $a\geqslant x$) for every $a\in A$, and \emph{order bounded} if it is order bounded both from above and below. A Riesz space is \textit{order complete} (or \emph{Dedekind complete}) if every non-empty subset $A\subseteq E$ that is order bounded from above has a supremum $\sup A$ (equivalently, if every non-empty subset that is order bounded from below has an infimum).
	
	For a net $(x_\alpha)_\alpha \subseteq E$, the notation $x_\alpha \uparrow x$ means that $(x_\alpha)_\alpha$ is increasing and that its supremum in $E$ is $x$;  we define $x_\alpha \downarrow x$  analogously. A net $(x_\alpha)_\alpha$ in $E$ \textit{converges in order} to some $x\in E$, denoted by $x_\alpha\xrightarrow{o} x$, if there is a net $(y_\alpha)_\alpha$ satisfying $y_\alpha\downarrow 0$ and $|x_\alpha-x|\leqslant y_\alpha$ for each $\alpha$. A set $A\subseteq E$  is \textit{order closed} if $(x_\alpha)_\alpha \subseteq A$ and $x_\alpha\xrightarrow{o} x$ imply $x\in A$. An order closed ideal of $E$ is called a \textit{band} of $E$. 
	If $A\subseteq E$,  its \textit{disjoint complement}  $A^d$ is defined by
	\begin{equation}\label{eq:DisjCompl}
		A^d:=\{x\in E: |x|\wedge |y|=0, \: \text{ for all } \: y\in A \}
	\end{equation}
	and it is a band. A band $A$ in a Riesz space $E$ is a \textit{projection band} if $E=A\oplus A^d$; the corresponding projections on $A$ and $A^d$ are called order projections, or band projections. 
	Every band  in an order complete Riesz space  is a projection band \cite[Th.\ 8.20]{AlBo99}.
	
	\begin{proposition}
		\label{th: dependence atomic part}
		If $\cG\subseteq\cF, \cG \neq \emptyset$ is closed under countable unions, then $\cA(\cG), \cD(\cG)\subseteq \cM$ are closed complemented and orthogonal subspaces, Chebyshev sets, and projection bands, each one the disjoint complement of the other.
		
		The  maps $A:\mu\mapsto\mu_\cG$ and $D:\mu\mapsto\mu_\cG^\perp$ are the  band projections of $\cM$ on $\cA(\cG)$ and $\cD(\cG)$, and they are positive linear contractions, unique nearest-point maps, and order-continuous homomorphisms of Riesz spaces, and satisfy $|A\mu|\leq |\mu|$, $|D\mu|\leq |\mu|$ for every $\mu\in \cM$.
	\end{proposition}

	Throughout our proofs, we  will sometimes denote with $G_\alpha\in \cG_\alpha$ a representative of the equivalence class $\cG_\alpha$ (for any $\alpha\in \cM$): while it is normally best to ignore the difference between the two (e.g.~to write $\cG_\mu \subseteq \cH_\mu$ $\mu$-a.e. to mean $G_\mu \subseteq H_\mu$ $\mu$-a.e., or equivalently $\cG_\mu \leq_\mu \cH_\mu$, i.e., $\cG_\mu \subseteq \cH_\mu$), we will occasionally be picky about this difference, when it makes things clearer. 
	
	\begin{proof}[Proof of Proposition \ref{th: dependence atomic part}.]
		Obviously two measures $\mu,\nu\in \cM$ are singular if and only if there exists $F\in \cF$ such that $|\mu|(F^c)=0=|\nu|(F)$, and this is equivalent to $|\mu| \wedge |\nu| =0$. 
		Thus $\nu \in \cD(\cG)$ if and only if $\nu$ is singular with every measure concentrated on a set of $\cG$, i.e., $\cD(\cG)=\cA(\cG)^d$; in particular $\cD(\cG)$ is a band.
		
		That $\cA(\cG)$ is a vector subspace follows simply from the fact that, if $\mu,\nu$ are concentrated on $G_\mu,G_\nu \in \cG$, then $\mu+\nu$ is concentrated on $G_\mu\cup G_\nu\in  \cG$, and if $t\in \bR$ then $ t\mu$ is concentrated on $G_\mu$. 
		Analogously, $\cD(\cG)$ is a vector subspace, since $|\mu|(G)=|\nu|(G)$ and $t\in \bR $ imply $|\mu+ \nu|(G)=0=|t \mu|(G)$. Since $\cA(\cG)$ and $\cD(\cG)$  are vector subspaces, $A,D$ are linear.
		If $\mu_n \to \mu$ and $\mu_n$ is concentrated on $G_n\in \cG$ then $\mu$ is concentrated on $G:=\cup_n G_n\in \cG$, and so $\cA(\cG)$ is (topologically) closed. If $\mu_n \to \mu$ then  $|\mu_n|(G)=0$  for every $n$ implies $|\mu|(G)=0$, and so $\cD(\cG)$ is (topologically) closed. 
		That $\cA(\cG)$ and $\cD(\cG)$ are complemented follows from Theorem \ref{thm: Dellacherie Thm Positive Meas}, since the existence of the decomposition for every $\mu\in \cM$ shows that $\cM=\cA(\cG)+\cD(\cG)$, and its uniqueness shows that $\cA(\cG) \cap \cD(\cG)=\{0\}$.
		Since $\cA(\cG)$, $\cD(\cG)$  are \emph{closed} complemented subspaces of $\cM$, then $A,D$ are continuous  (see \cite[Th.\ 6.47 ]{AlBo99}). 
		
		Since $\mu\geq 0$ implies $\mu_\cG, \mu_\cG^\perp \geq 0$,  then $ A,D$ are positive; since they are linear, they are increasing.
		To prove that $\cD(\cG)^d=\cA(\cG)$, recall that in every order complete Riesz space every band $B$ satisfies $(B^{d})^d=B$ (see \cite[Lem.\ 8.4, Th.\ 8.19]{AlBo99}); since by definition $\cD(\cG)=\cA(\cG)^d$, to conclude that $\cD(\cG)^d=\cA(\cG)$ we only need to show that $\cA(\cG)$ is a band (and thus $\cD(\cG), \cA(\cG)$ are projection bands). Since it is clearly an ideal of $\cM$, we only need to show that  $0 \leq \mu^{i} \uparrow \mu$ and $\{\mu^{i}\}_{i\in I} \subseteq \cA(\cG)$ imply $\mu \in \cA(\cG)$,  as this implies that $\cA(\cG)$  is order closed (see \cite[First 6 lines of Sec.\ 8.9]{AlBo99}) and thus a band. 
		Since $\nu\in \cM$ is in $\cA(\cG)$ if and only if $\nu=\nu_\cG$, and $A,D$ are increasing,  we can equivalently show that $0 \leq \mu^{i} \uparrow \mu$  implies $\vee_{i}\mu^{i}_\cG  =\mu_\cG$ (and  so $0 \leq \mu^{i}_\cG \uparrow \mu_\cG$). 
		So, assume $0 \leq \mu^{i} \uparrow \mu$; as  $\mu\geq \mu^i$ implies  $\mu_\cG\geq\mu^i_\cG$, then $\{\mu^i_\cG\}_{i\in I}$ is bounded from above, so it admits a supremum, which by definition satisfies 
		\begin{equation}
			\mu_\cG\geq\vee_{i\in I} \mu^i_\cG .
		\end{equation}
		To prove the opposite inequality, since $G_{\mu^i}$ is the $\leq_{\mu^i}$-maximum of $\cG$  and $G_\mu\in\cG$, we have that
		\begin{equation}
			\mu^i_\cG(F)=\mu^i(F\cap G_{\mu^i})\geq\mu^i(F\cap G_\mu), \quad \text{ for all } \: F\in\cF, 
		\end{equation}
		so that 
		\begin{equation}
			(\vee_{j\in I}\mu^j_\cG)(F)\geq \mu^i_\cG(F)\geq \mu^i(F\cap G_\mu), \quad \text{ for all } \: F\in\cF ,
		\end{equation}
		and taking the supremum over $i\in I$ and using \eqref{IncreasingNet} we conclude that, for every $F\in \cF$, 
		\begin{equation}
			(\vee_{i\in I}\mu^i_\cG)(F)\geq \sup_{i\in I} (\mu^i(F\cap G_\mu))=(\vee_{i\in I}\mu^i)(F\cap G_\mu)=\mu(F\cap G_\mu)=\mu_\cG(F),
		\end{equation}
		proving that indeed $\vee_{i\in I} \mu^i_\cG= \mu_\cG$, and so $\mu^i_\cG \uparrow \mu_\cG$.
		
		If $P$ is any projection on a projection band of a Riesz space $R$, then $P$ is an  order-continuous homomorphism of Riesz spaces, and satisfies $|Pr|\leq |r|$ for every $r\in R$, and so such are $A,D$; the inequalities  $|A\mu|\leq |\mu|$, $|D\mu|\leq |\mu|$ trivially hold anyway, and   they imply that $A,D$ are contractions.
		Since $|\mu| \wedge |\nu|=0 $  implies $|\mu + \nu|=|\mu |+ |\nu|$ (see \cite[Th.\ 8.12]{AlBo99}), and the norm on $\cM$ is given by $\|\mu\|_\cM=|\mu|(\Omega)$, we get that  $\|\mu + \nu\|_\cM=\|\mu \|_\cM+ \|\nu\|_\cM$ for every $\mu \in \cA(\cG), \nu \in \cD(\cG)$, and so
		$A,D$ are the nearest point projections on  $\cA(\cG), \cD(\cG)$, and these are Chebyshev sets which are orthogonal. 
	\end{proof}

	We now  study the order properties of  $\cG\mapsto\mu_\cG$ and $\cG\mapsto \cG_\mu$. 
	Let $\bS:=\{\cG: \cG\subseteq\cF,\cG\neq\emptyset \}$ be the set of all non-empty subfamilies of $\cF$, and define its subset
	\begin{equation}
		\bG:=\{\cG\in\bS: (G_n)_{n\in\bN}\subseteq\cG \: \Rightarrow \: \cup_{n\in\bN} G_n\in\cG \}
	\end{equation}
	of all subfamilies of $\cF$ which are closed under countable unions. We endow $\bS$ and $\bG$ with the order of inclusion of sets.
	Let $\cG\in\bS$. Since the intersection of families closed under countable unions is closed under countable unions, the intersection of every $\cH\in\bG$ such that $\cG\subseteq\cH$ is the smallest element of $\bG$ that contains $\cG$, which we denote by $\cG^\sigma$. Notice that $\cG^\sigma$ can also be built `from the inside',  as 
	\begin{align}
		\label{eq: built from inside}
		\cG^\sigma=\{\cup_{n\in \bN} G_n:  G_n\in\cG, n\in \bN \},
	\end{align}
	and 
	\begin{equation}
		\bG=\{\cG\in\bS: \cG=\cG^\sigma \}.
	\end{equation}
	For an arbitrary family $\{\cG^i \}_{i\in I}\subseteq\bG$, its infimum $\wedge_{i\in I} \cG^i$ and supremum $\vee_{i\in I}\cG^i$ in $\bG$ are given by
	\begin{align}
		\label{eq: sup is sigma union}
		\wedge_{i\in I}\cG^i &=\cap_{i\in I}\cG^i\\
		\vee_{i\in I}\cG^i&=(\cup_{i\in I}\cG^i)^\sigma.
	\end{align}
	Since \eqref{eq: built from inside} expresses $\cG^\sigma$ as the family of countable unions of sets in $\cG$,  we get that
	\begin{equation}\label{UnionSigma}
		(\cup_{i\in I}\cG^i)^\sigma=\{\cup_{i\in J} G_i:  J\subseteq I \text{ is countable, } J\neq\emptyset \text{ and } G_i\in\cG^i \},
	\end{equation}
	and so
	\begin{equation}\label{VeeGiSigma}
		\vee_{i\in I}\cG^i=\{\cup_{i\in J} G_i:  J\subseteq I \text{ is countable, } J\neq\emptyset \text{ and } G_i\in\cG^i \}.
	\end{equation}
	For instance, if $\cG,\cH\in\bG$ then
	\begin{equation}
		\cG\vee\cH=\{G,H,G\cup H: G\in\cG, H\in\cH \};
	\end{equation}
	and in particular if $\emptyset \in \cG \cap \cH$ then 
	$\cG\vee\cH=\{G\cup H: G\in\cG, H\in\cH \}$.

	\begin{proposition}\label{Prop:GmapGmu} If $\mu \in \cM$, 
		the map $\cG\mapsto\cG_\mu$ is increasing, i.e., if $\cG\subseteq\cH$ then $\cG_\mu\subseteq\cH_\mu$, $\mu$-a.e.~and it satisfies the following property: if $\{\cG^i \}_{i\in I}\subseteq\bG$, then there exists $J\subseteq I$ countable such that
		\begin{equation}\label{VeeGmu}
			(\vee_{i\in I} \cG^i)_\mu=(\vee_{i\in J} \cG^i)_\mu.
		\end{equation}
	Moreover, $(\vee_{i\in J} \cG^i)_\mu=\cup_{i\in J}\cG^i_\mu$ holds for any  countable $J\subseteq I$.
	\end{proposition}
	\begin{proof} Since $\cG_\mu=\cG_{|\mu|}$, we can assume w.l.o.g.~that $\mu \in \cM_+$. 
		Monotonicity is straightforward: as $\cH_\mu$ is the $\leq_{\mu}$-maximum of $\cH$ and $\cG_\mu\in\cG\subseteq\cH$, then also $\cG_\mu\subseteq\cH_\mu$, $\mu$-a.e..	
		
	We now show that $(\vee_{i\in J} \cG^i)_\mu=\cup_{i\in J}\cG^i_\mu$ holds for any  \emph{countable} set $J$. 
		Indeed \eqref{VeeGiSigma} shows that $\cup_{i\in J} \cG^i_\mu\in\vee_{i\in J}\cG^i$, and that if 	 $G\in \vee_{i\in J} \cG^i$ then  $G=\cup_{i\in J_1} G_i$ for some $J_1\subseteq J$, $J_1\neq\emptyset$ and $G_i\in\cG^i$. Therefore,
		$G \subseteq \cup_{i\in J_1} \cG^i_\mu \subseteq \cup_{i\in J} \cG^i_\mu$, $\mu$-a.e., so $\cup_{i\in J} \cG^i_\mu$ is the $\leq_{\mu}$-maximum of $\vee_{i\in J}\cG^i$.	

Now we prove that there exists $J\subseteq I$ countable such that $(\vee_{i\in J} \cG^i)_\mu \subseteq (\vee_{i\in I} \cG^i)_\mu \subseteq \cup_{i\in J}\cG^i_\mu$, which concludes the proof. By \eqref{VeeGiSigma}, since the set  $(\vee_{i\in I}\cG^i)_\mu$ belongs to   $\vee_{i\in I}\cG^i$, it is of the form $\cup_{i\in J} G_i$ for some $G_i\in\cG^i$ and countable $J\subseteq I$, $J\neq\emptyset$. Since $\cup_{i\in J} G_i\subseteq \cup_{i\in J} \cG^i_\mu $, $\mu$-a.e., we obtain
$(\vee_{i\in I}\cG^i)_\mu \subseteq \cup_{i\in J} \cG^i_\mu$. 
Finally, given any $J\subseteq I$,  $\vee_{i\in J} \cG^i \subseteq \vee_{i\in I} \cG^i$ implies $(\vee_{i\in J} \cG^i)_\mu \subseteq (\vee_{i\in I} \cG^i)_\mu$
since $\cG\mapsto\cG_\mu$ is increasing.
	\end{proof}

	\begin{proposition}
		If $\mu \in \cM_+$, the map $\cG\in\bG\mapsto\mu_\cG\in\cM_+$	is
		\begin{enumerate}[(i)]
			\item increasing, i.e., if $\cG\subseteq\cH$ then $\mu_\cG\leq\mu_\cH$;
			\item continuous from below with respect to the order, i.e., if $\{\cG^i \}_{i\in I}\subseteq \bG$ then
			\begin{equation}\label{MuVeeG}
				\mu_{\vee_{i\in I}\cG^i}=\vee_{i\in I}\mu_{\cG^i}.
			\end{equation}
		\end{enumerate}
	\end{proposition}
	\begin{proof}
		To prove monotonicity, recall that $\mu_\cG(\cdot):=\mu(\cdot\cap \cG_\mu)$ and $\mu_\cH(\cdot):=\mu(\cdot\cap\cH_\mu)$ with $\cG_\mu\in\cG$ and $\cH_\mu\in\cH$. Moreover, $\cH_\mu$ is the maximum in $\cH$ and so $\mu(F\cap H)\leq \mu(F\cap \cH_\mu)$ for every $F\in\cF$ and $H\in\cH$, and in particular for $H=\cG_\mu\in\cG\subseteq\cH$; thus $\mu_\cG\leq\mu_\cH$.
		
		Let us prove \eqref{MuVeeG}. By the  monotonicity of $\cG\mapsto\mu_\cG$ we  have that 
		$$\mu_{\vee_{i\in I}\cG^i}\geq \vee_{i\in I}\mu_{\cG^i}=:\nu ;$$ let us prove the opposite inequality. Notice that, since $\nu \ll \mu$ holds (e.g.~because $\nu \leq \mu$), we do not need to be careful about $\mu$-null sets, so we can improperly write $F\subseteq\cG_\mu^i$ etc.~without causing problems. By Proposition \ref{Prop:GmapGmu}, there exists  $J\subseteq I$  countable and such that $(\vee_{i\in I}\cG^i)_\mu=\cup_{i\in J}\cG^i_\mu$; let $(j_n)_{n\in\bN}$ be an enumeration of $J$.  Since
		$$\mu_{\vee_{i\in I}\cG^i}(\cdot):=\mu(\cdot\cap(\vee_{i\in I}\cG^i)_\mu)=\mu(\cdot\cap(  \cup_{i\in J}\cG^i_\mu )),$$ 
		showing $\nu\geq\mu_{\vee_{i\in I}\cG^i}$ is equivalent to showing that $\nu(F)\geq \mu(F)$ for every $F\subseteq \cup_{i\in J}\cG^i_\mu $, $F\in\cF$.  Notice that we can write an arbitrary $F\subseteq (\vee_{i\in I}\cG^i)_\mu=\cup_{n\in\bN}\cG^{j_n}_\mu$, $F\in\cF$ as the countable union of disjoint sets $F^n\subseteq\cG^{j_n}_\mu$, $n\in\bN$ defined as
		\begin{equation}
			F^0:=F\cap\cG^{j_0}_\mu, \quad F^{n+1}:=(F\cap\cG^{j_{n+1}}_\mu)\setminus(\cup_{k=1}^n\cG^{j_k}_\mu), \quad n\in\bN.
		\end{equation}
		Hence, since $\nu\geq\mu_{\cG^i}=\mu(\cdot \cap \cG^i_\mu)$ gives $\nu(G)\geq\mu(G)$ for all $G \subseteq \cG^i_{\mu}, G \in \cF$, we get 
		\begin{equation}
			\nu(F)=\sum_{n=1}^\infty\nu(F^n)\geq \sum_{n=1}^\infty\mu(F^n)=\mu(F),
		\end{equation}
		and the proof is complete.
	\end{proof}

	\section{The decomposition theorem on Riesz spaces}
	\label{sec: dec in riesz spaces}
	In this section we point out  that Theorem \ref{thm: Dellacherie Thm Positive Meas} can be further generalised to a decomposition result for operators on Riesz spaces. 
	
	To this purpose, we will make use of the definitions and facts about Riesz spaces described  in Section \ref{se: dependence of decomposition}, and a few more concepts which we now introduce. We will rely on results from \cite{aliprantis2006positive}; we recall that every Riesz space considered in \cite{aliprantis2006positive} is assumed to be Archimedean (see \cite[page 9]{aliprantis2006positive}). This is without loss of generality in our case since  we only consider order-complete Riesz spaces and every order-complete Riesz space is Archimedean \cite[Lem.\ 8.4]{AlBo99}.  Let $(E,\geqslant)$ be an order-complete Riesz space and $A\subseteq E$. Recall, from  \eqref{eq:DisjCompl}, that $A^d$ denotes the disjoint complement of $A$. Then, by \cite[Th.\ 1.39]{aliprantis2006positive}, $(A^d)^d$ is the smallest (with respect to the inclusion of sets) band of $E$ that contains $A$, it is called the band generated by  $A$, and it is denoted by $\text{Band}(A)$; if $A=\{x\}$ we write $\text{Band}(x)$ for $\text{Band}(\{x\})$. 
	
	The (real) vector space of all operators (i.e., linear maps) from $E$ to $F$ will be denoted by $\mathcal{L}(E, F)$. We say that $T\in \mathcal{L}(E, F)$ is \emph{positive} (and write $T\geq 0$) if $T(x)\geq 0$ for any $x\in E$ such that $x\geq 0$, and it is \emph{strictly positive} if $T(x)>0$ for any $x\in E$ such that $x>0$.
	An operator $T\in \mathcal{L}(E, F)$  is said to be \emph{order bounded} if it maps order bounded subsets of $E$ to order bounded subsets of $F$; the set of order bounded operators $T\in \mathcal{L}(E, F)$ is denoted by $\mathcal{L}_b(E, F)$. 
	Any positive operator is trivially order bounded; thus the difference of positive operators is  order bounded. Conversely, if  $F$ is order complete then  (see \cite[Th.\ 1.18]{aliprantis2006positive})   $\mathcal{L}_b(E, F)$ is an order complete Riesz space under its usual ordering ($T \geq S$ if $T(x) \geq S(x)$ holds for all $x\in E_+$), and so one can write any $T\in \mathcal{L}_b(E, F)$ as $T=T^+ - T^-$ with $T^\pm$ positive operators.  
	
	If $T:E\to F$ is a positive operator, $A$ is an ideal in $E$ and $F$ is order complete, one can define
	$$T_A(x):=\sup \{T(y): y \in A, 0 \leq y \leq x\}  \text{ for all } x \in E_+, $$
	and then extend $T_A$ to $E$ by setting
	$$T_A(x):=T_A(x^+) - T_A(x^-) \text{ for all } x \in E.$$
	Then, by combining \cite[Th.\ 1.28]{aliprantis2006positive} and \cite[Th.\ 1.10]{aliprantis2006positive}, it follows that $T_A:E\to F$ is a positive operator. Thus, if $T\in \mathcal{L}_b(E, F)$, the map
	\begin{align}
		\label{eq: def T_A for any T}
		T_{A}(x):=(T^+)_A(x)-(T^-)_A(x), \quad x\in E
	\end{align} 
	satisfies $T_{A}\in \mathcal{L}_b(E, F)$, and the map $T \mapsto T_{A}$ from $\mathcal{L}_{\mathrm{b}}(E, F)$ to $\mathcal{L}_{\mathrm{b}}(E, F)$ is a positive operator and an order projection (see \cite[Th.\ 2.7 and the preceding discussion]{aliprantis2006positive}).

	We now show that if $A$ is a band and $E$ is order complete, $T_A$ admits a convenient expression; while this result is very likely known, we could not find it in either of our standard references \cite{AlBo99,aliprantis2006positive}. Recall that if $A$ is a projection band in $E$ (i.e., a band such that $E=A\oplus A^d$), the \textit{projection map} $P_A: E\to E$ is defined by $P_A(x):=x_1$, where $x=x_1+x_2$ with $x_1\in A$ and $x_2\in A^d$.
	\begin{lemma}
		\label{prop: T_A=TP_A}
		Let $E$, $F$ be order complete Riesz spaces, $A\subseteq E$ be a band, and $T\in \mathcal{L}_b(E, F)$ then 
		$$T_A=T\circ P_A, \quad T_{A^d}=T\circ P_{A^d}, \quad T_A + T_{A^d}=T, $$
		and if $T\geq 0$ then $T_A \wedge T_{A^d}=0, \quad T_A \vee T_{A^d}=T$. 
	\end{lemma}
	\begin{proof}
		Since every band  in an order complete Riesz space  is a projection band \cite[Th.\ 8.20]{AlBo99}, $P_{A}$ exists and $P_{A}+P_{A^d}=I$ holds, and so $T\circ P_A + T\circ P_{A^d}=T$. 
		By \cite[Th.\ 1.28]{aliprantis2006positive} if $T\geq 0$ then $T_{A}=T$ on $A$ and $T_{A}=0$ on $A^{d}$, and so $T_A(x)=T\circ P_A(x)$ holds for all $x\in A$ and for all $x\in A^{d}$, and so by linearity for all $x\in E$. 
		Using \eqref{eq: def T_A for any T} and the linearity of $T \mapsto T \circ P_A$, one obtains that 
		$$T_{A}=(T^+)_A - (T^-)_A=T^+ \circ P_A - T^- \circ P_A=T \circ P_A$$
		holds also for any $T\in \mathcal{L}_b(E, F)$. 
		Since $A^d$ is a band this shows also  that $T_{A^d}=T\circ P_{A^d}$.  For the identities  $T_A \wedge T_{A^d}=0$ and $T_A \vee T_{A^d}=T$ see  \cite[Sec.\ 1.2, Ex.\ 3]{aliprantis2006positive}.
	\end{proof}
	
	We can now introduce the decomposition theorem for positive operators on Riesz spaces which generalises Theorem \ref{thm: Dellacherie Thm Positive Meas}. Following \cite{aliprantis2006positive}, we say that $E$ has the \emph{countable sup property} if, for any  $D\subseteq E$ which admits a supremum,  there exists an at most countable set $C\subseteq D$ that admits a supremum and it satisfies $\sup C=\sup D$. 
	\begin{theorem}\label{thm:decomp-Riesz}
		Let $E$, $F$ be order complete Riesz spaces,  $T:E\to F$ a strictly positive operator, $ \emptyset \neq B\subseteq E_+$, $A:=\text{Band}(B)$, and assume $B$ is order bounded from above. 	Then 
		\begin{enumerate}
			\item \label{it: generalisation} The following are equivalent
			\begin{enumerate}
				\item \label{it: x maximiser of B} There exists $x\in B$ such that $B \subseteq \text{Band}(x)$. 
				\item  \label{it: decomp T=U+V} There exist $z\in B$ and $U,V:E\to F$ positive operators  such that 
				\begin{align}
					\label{eq: T=U+V}
					T=U+V, \quad U(x)=0  \text{ for all } x\in (\text{Band}(z))^d , \quad V(y)=0  \text{ for all } y\in A .
				\end{align}
			\end{enumerate}
			Moreover, if the above statements hold then 
			$$A=\text{Band}(x)=\text{Band}(z),  \quad  U=T_{A}, \quad  V=T_{A^d}.$$
			It follows that such $U,V$ are unique and satisfy $U \wedge V =0$ and $U \vee V =T$.	
			\item \label{it: any B sigma closed has max} If $F$ has the countable sup property and $\{b_n\}_{n\in \bN}\subseteq B$ implies $\sup_n b_n \in B$ then there exists $x\in B$ as in item \ref{it: x maximiser of B}.	
		\end{enumerate}	
		
	\end{theorem}
	
	The proof of Theorem \ref{thm:decomp-Riesz} requires the following lemma.
	
	\begin{lemma}
		\label{prop: strictly positive w}
		If $E$ is an order complete Riesz space and $z\in B\subseteq E_+$ is such that $A:=\text{Band}(B)\neq \text{Band}(z)=:Z$ then there exists $w\in A \cap Z^d$ such that $w>0$.
	\end{lemma}
	\begin{proof}
		Trivially $Z\subset A$, so there exists $y \in A\setminus Z$. Since $A$ is a band $y^\pm \in A$, and since $Z$ is a vector space $y^+$ and $y^-$ cannot both belong to it; thus, by possibly replacing $y$ with $-y$,  we can assume w.l.o.g.~that $y^+\notin Z$, so  $y^+ \in A\setminus Z$. Since $E$ is complete, $Z$ is a projection band \cite[Th.\ 8.20]{AlBo99}. As any order projection is a positive operator \cite[Th.\ 1.43]{aliprantis2006positive} $y^+\geq 0$ implies $w:=P_{Z^d}(y^+)\geq 0$, and $Z \not \ni y^+=P_Z(y^+)+w $ implies $w \notin Z \ni 0$ and so $w>0$. Trivially $w\in Z^d$, and since $y^+\in A$ and $P_Z(y^+)\in Z \subseteq A$ we get $w=y^+ - P_Z(y^+)\in A$.
	\end{proof}
	
	\begin{proof}[Proof of Theorem \ref{thm:decomp-Riesz}]
		If item \ref{it: x maximiser of B} holds then $A= \text{Band}(x)$, and so Lemma \ref{prop: T_A=TP_A} gives that  item \ref{it: decomp T=U+V} holds with $z:=x, U:=T_{A},   V:=T_{A^d}$.
		Conversely, assume that item \ref{it: decomp T=U+V} holds. Since $z\in B$ trivially implies $A^d \subseteq (\text{Band}(z))^d$, we get that $U=T\circ P_A, V=T\circ P_{A^d}$ hold both on $A$ and on $A^d$, and so on  $A+A^d=E$, and so Lemma \ref{prop: T_A=TP_A} gives $U=T_{A},   V=T_{A^d}$ (in particular $U,V$ are unique) and $U \wedge V =0, U \vee V =T$. Finally $A=\text{Band}(z)$ holds, otherwise Lemma \ref{prop: strictly positive w} yields a $w>0$ such that $T(w)=U(w)+V(w)=0$, contradicting the strict positivity of $T$. In particular taking $x:=z$ we get that $B\subseteq A=\text{Band}(x)$, so item \ref{it: x maximiser of B} holds. 
		
		We now prove item \ref{it: any B sigma closed has max}.
		Since $E$ is order complete,  $B\neq \emptyset$ is  order bounded from above there exists $ \sup B=:x\in E$. 
		Since every order complete Riesz space is Archimedean \cite[Lem.\ 8.4]{AlBo99},  $E$ has the countable sup property \cite[Ch.\ 1.4, Ex.\ 1]{aliprantis2006positive}, and so there exists $\{b_n\}_{n\in \bN}\subseteq B$ such that $\sup_n b_n=x$. The assumption made on $B$ implies that $x\in B$ and, by definition of supremum, $x$ satisfies $y\leq x$ for all $y\in B$; since $B\subseteq E_+$ and $\text{Band}(x)$ is a band and so it is solid, it follows that $B\subseteq \text{Band}(x)$.
	\end{proof}

	\begin{remark}
		It is natural to wonder when one can replace $A$ with $B$ in \eqref{eq: T=U+V}. This can be done if $V:E\to F$ is an onto order-continuous lattice homomorphism, since in this case the kernel of $V$ is a band \cite[Th.\ 2.21, part (2)]{aliprantis2006positive}, and so $V=0$ on $B$ automatically implies $V=0$ on $A$.
		Recall also that a bijective operator $V:E\to F$ is a lattice homomorphism if and only if $V$ and $V^{-1}$ are positive \cite[Th.\ 9.17]{AlBo99}.
	\end{remark}

	\begin{remark}
		\label{rem: band A decomp}
		Theorem \ref{thm:decomp-Riesz} considers the decomposition $T=T_A + T_{A^d}$ of Lemma \ref{prop: T_A=TP_A} in the special case in which $A$ is a \emph{principal band}, i.e., it is generated by one element. This prompts the question of how general this is, i.e., when it is the case that all bands are principal bands. It turns out  that this is the case whenever $E$ admits a weak unit $e\in E$ (i.e., an element $e\in E$ such that $\text{Band}(e)=E$), that is, if $E$ is itself a principal band (see \cite[Ch.\ 4, Th.\ 31.1(ii)]{Lux00}). Notice that this will cover the setting of Example \ref{ex: T=mu}, since any strictly positive constant is a weak unit in $L^p(\mu)$.
	\end{remark}

	We will soon explain how Theorem \ref{thm:decomp-Riesz} subsumes Theorem  \ref{thm: Dellacherie Thm Positive Meas}, using the following Lemma. 
	In Lemma \ref{prop:band-L^p} and Example \ref{ex: T=mu},  $\mu$ will be a finite, positive,  $\sigma$-additive measure on the measurable space $(\Omega,\cF)$.
	\begin{lemma}\label{prop:band-L^p}
		Given  $A\in \cF$, define 
		$$O_\mu(A):=\{g\in L^p(\mu) | \: g=0 \: \mu\text{-a.e.\ on } A \}.$$
		Then, the  band generated by any $f\in L^p(\mu)$ and its disjoint complement in $L^p(\mu)$ are given by
		$$\text{Band}(f)=O_\mu(\{f=0 \}),  \quad (\text{Band}(f))^d=O_\mu(\{f \neq 0 \}).$$
	\end{lemma}
	\begin{proof}
		From \cite[Th.\ 1.38]{AlBo99}, we have that
		$$\text{Band}(f)=\{g\in L^p(\mu) | \: (|g|\wedge n|f|)_{n\in \bN}\uparrow |g| \}.$$
		Since $|g|\wedge n|f|\uparrow |g| \ind_{\{f\neq 0 \}}$, we get that $\text{Band}(f)=O_\mu(\{f=0 \})$,
		and thus trivially $O_\mu(\{f \neq 0 \})\subseteq (\text{Band}(f))^d$. 
		Conversely,  $\ind_{\{f\neq 0 \}}\in \text{Band}(f)$, so  if $g\in (\text{Band}(f))^d$ we get so $0=|g|\wedge \ind_{\{f\neq 0 \}}$, which implies $g\in O_\mu(\{f \neq 0 \})$. 
	\end{proof}
	
	\begin{example}
		\label{ex: T=mu}	
		We now illustrate how Theorem \ref{thm:decomp-Riesz} generalises Dellacherie's decomposition  \ref{thm: Dellacherie Thm Positive Meas}. 
		Fix $p\in [1,\infty]$ and consider the order-complete Riesz spaces  $E:=L^p(\mu), F =\bR$, and the positive operator $T:E\to F, f \mapsto \int f d\mu$.  As usual in measure theory we identify any $F\in \cF$ with its indicator function $\ind_F$, so $\cG\subseteq\cF$ is identified with $B\subseteq E_+$ given by $B:=\{\ind_G: G\in \cG\}$; correspondingly $\mu$ is identified with the operator $f \mapsto \int f d \mu$. 
		Clearly $\bR$ has the countable sup property, and if $\cG$ is closed under countable unions (even if only up to $\mu$-null sets) then $\{b_n\}_{n\in \bN}\subseteq B$ implies $\sup_n b_n \in B$. 
		By Lemma \ref{prop:band-L^p}, $G_\mu\in\cG$ (i.e., $\ind_{G_{\mu}}\in B$) is such that $G \setminus G_\mu$ is a $\mu$-null set for every $G\in \cG$ iff $B \subseteq  \text{Band}(\ind_{G_{\mu}})$.  By Lemma \ref{prop:band-L^p}, if $\bar{G}\in \cG$ and $A:=\text{Band}(\ind_{\bar{G}})$ then $P_A(f)=f \ind_{\bar{G}}$ for all $f\in L^p(\mu)$, so by Lemma \ref{prop: T_A=TP_A}  
		$$ \textstyle T_A(f)=\int f d\mu(\bar{G}\cap \cdot) \quad \text{ and } \quad  T_{A^d}(f)=\int f d\mu(\bar{G}^c\cap \cdot) \quad \text{ for all } f\in L^p(\mu).$$ 
	\end{example}
	
	\begin{remark}
		\label{rem: T to TA order projection}
		As  mentioned just after \eqref{eq: def T_A for any T}, the map $T \mapsto T_{A}$ from $\mathcal{L}_{\mathrm{b}}(E, F)$ to $\mathcal{L}_{\mathrm{b}}(E, F)$ is an order projection; this fact, combined with Example \ref{ex: T=mu}, shows that $\mu \mapsto \mu_{\cG}$ and $\mu \mapsto \mu_{\cG}^\perp$ are band projections on $\cM$. 	Proposition \ref{th: dependence atomic part} elaborates on this result, and in particular considers  not just the order structure but also the metric structure on $\cM$.  
	\end{remark}
	
	\begin{remark}
		The Hewitt-Yosida decomposition of a positive measure $\mu$ into its $\sigma$-additive and its purely finite additive component, which are disjoint, is an interesting example of a decomposition which does not arise in the setting of Remark \ref{rem: band A decomp}. 
		More generally, a decomposition of a measure into two disjoint components is associated to any band $B\subseteq \cM$: since $\cM$ is complete, any band in $\cM$ is a projection band, and so one one always write $\mu=P_B(\mu) + P_{B^d}(\mu)$ with $P_B \wedge P_{B^d}=0$. Instead, Remark \ref{rem: band A decomp} considers only the setting in which $B=(A^\circ)^d$ for some band $A\subseteq E$, where $A^\circ$ denotes the \emph{annihilator} of $A$, defined as $$A^{\mathrm{o}}:=\left\{T \in \mathcal{L}_{\mathrm{b}}(E, F): T=0 \text{ on } A\right\};$$
		indeed, if $B=(A^\circ)^d$  then $P_B(T)=T_A$ for each $T\in \mathcal{L}_{\mathrm{b}}(E, F)$, see \cite[Ch.\ 2.1, Ex.\ 4]{aliprantis2006positive}.
	\end{remark}

	\bibliography{bib_for_paper}{}
	\bibliographystyle{abbrv} 
	
\end{document}